\documentclass[12pt]{amsart}

\usepackage{pict2e}
\usepackage{amssymb}
\usepackage{graphicx}
\usepackage{amsmath}
\usepackage{amscd}
\usepackage{latexsym}
\usepackage{tikz}
\usetikzlibrary[shapes,arrows,snakes]
\tikzstyle{bull}=[circle,draw=black,fill=black!80]
\tikzstyle{holl}=[circle,draw=black]
\include{amsfonts}
\newtheorem{thm}{Theorem}[section]
\newtheorem{lemma}[thm]{Lemma}
\newtheorem{defn}[thm]{Definition}
\newtheorem{rmk}[thm]{Remark}
\newtheorem{coro}[thm]{Corollary}
\newtheorem{examp}[thm]{Example}
\newtheorem{prop}[thm]{Proposition}
\newtheorem{notation}[thm]{Notation}
\newtheorem{claim}[thm]{Claim}

\newtheorem{assumption}[thm]{Assumption}
\newtheorem{conjecture}[thm]{Conjecture}
\newtheorem{problem}[thm]{Problem}

\newcommand{\ts}{\textsc}
\newcommand{\bt}{$\bullet$}

\begin{document}

\title{Automorphisms of Decompositions}
\author{{\bf Tim Hannan and John Harding}}
\date{}

\maketitle
\markboth{Tim Hannan and John Harding}{Automorphisms of Decompositions}

\begin{abstract}
In \cite{Harding1} Harding showed that the direct product decompositions of many different types of structures, such as sets, groups, vector spaces, topological spaces, and relational structures, naturally form orthomodular posets. When applied to the direct product decompositions of a Hilbert space, this construction yields the familiar orthomodular lattice of closed subspaces of the Hilbert space. 

In this note we consider orthomodular posets $\ts{Fact}~X$ of decompositions of a finite set $X$. We consider the structure of these orthomodular posets, such as their size, shape, and connectedness, states, and begin a study of their automorphism groups in the context of the natural map $\Gamma$ from the group of permutations of $X$ to the automorphism group of $\ts{Fact}~X$. 

We show $\Gamma$ is an embedding except when $|X|$ is prime or 4, and completely describe the situation when $|X|$ has two or fewer prime factors, when $|X|=2^3$ and when $|X|=3^3$. The bulk of our effort lies in a series of combinatorial arguments to show $\Gamma$ is an isomorphism when $|X|=27$. We conjecture that this is the case whenever $|X|$ has sufficiently many prime factors of sufficient size, and hope that our arguments here might be adapted to the general case. 
\end{abstract}

\section{Introduction}

A binary decomposition of a structure $A$ consists of structures $B,C$ and an isomorphism $f:A\to B\times C$. Another binary decomposition $f':A\to B'\times C'$ is equivalent to the given one if there are isomorphisms $u:B\to B'$ and $v:C\to C'$ with $f'=(u\times v)\circ f$ as shown below. 

\setlength{\unitlength}{.05in}
\begin{center}
\begin{picture}(30,20)(0,0)
\put(0,7.5){\makebox(0,0)[r]{$A$}}
\put(25,15){\makebox(0,0)[c]{$B\times C$}}
\put(25,0){\makebox(0,0)[c]{$B'\times C'$}}
\put(2,8){\vector(8,3){16}}
\put(2,6){\vector(8,-3){16}}
\put(21.5,12){\vector(0,-1){9}}
\put(28.25,12){\vector(0,-1){9}}
\put(10,12.5){\makebox(0,0)[br]{$f$}}
\put(10,1.25){\makebox(0,0)[tr]{$f'$}}
\put(20,7.5){\makebox(0,0)[r]{$u$}}
\put(30,7.5){\makebox(0,0)[l]{$v$}}
\end{picture}
\end{center}
\vspace{3ex}

This equivalence of decmpositions is different than that encountered in the Kr\"{u}ll-Schmidt theorem, where only the isomorphism classes of the factors is important. In the above definition of equivalence, the order of the factors and the way the isomorphism $f$ decomposes the structure matter. This is analogous to the situation for onto homomorphisms $f:A\to B$, where equivalence could mean isomorphism of the images, or the existence of an isomorphism between the images compatible with the homomorphisms. While the second approach is more frequent when dealing with onto homomorphisms, it is the first usually encountered when dealing with decompositions. Our focus is the more refined equivalence for decompositions. 

A small example may help. A 4-element set $X=\{a,b,c,d\}$ has 8 equivalence classes of decompositions. One has as its representatives decompositions of $X$ as a product of a 4-element set and a 1-element set. There are a proper class of such decompositions, but all are equivalent. Another has as its representatives decompositions of $X$ as a product of a 1-element set and a 4-element set. There are 6 equivalence classes whose representatives decompose $A$ as a product of two 2-element sets. To see these can be different, consider one bijection $f$ from $X$ to a product of two 2-element sets where $f(a)$ and $f(b)$ have the same first component, and another bijection $g$ from $X$ to a product of two 2-element sets where $g(a)$ and $g(b)$ have different first components. The decompositions given by $f$ and $g$ will be inequivalent. 

\begin{defn}
For a structure $A$, let $\ts{Fact}~A$ be the set of all equivalence classes of binary decompositions of $A$. 
\end{defn}

When dealing with a set $X$, each equivalence class of decompositions has a unique representative $f:X\to X/\theta_1\times X/\theta_2$ where $\theta_1$ and $\theta_2$ are equivalence relations on $X$ and the natural map $f$ given by these relations is a bijection. Such pairs of equivalence relations are called factor pairs \cite{Burris}. 

\begin{thm}
For a set $X$, $\ts{Fact}~X = \{(\theta_1,\theta_2):(\theta_1,\theta_2)$ is a factor pair$\,\}$. 
\end{thm}

We next consider the matter of defining structure on the set $\ts{Fact}~A$. We require the following well-known notion from the study of quantum logic \cite{Kalmbach,Ptak}. 

\begin{defn}
An orthomodular poset (abbreviated: \ts{omp}) is a bounded poset $P$ with bounds $0,1$ and a unary operation $'$ that is order inverting and period two, such that the following conditions hold with regard to the existence and behavior of certain joins and meets. Here, $x\perp y$ means $x\leq y'$ and is read ``$x$  is orthogonal to $y$.''
\vspace{1ex}
\begin{enumerate}
\item For each $x\in P$, $x\wedge x'$ exists and is $0$, and $x\vee x'$ exists and is $1$. 
\item If $x\perp y$, then the join $x\vee y$ exists. 
\item If $x\perp y'$, then $x\vee (x\vee y)' = y'$.
\end{enumerate}
\end{defn}

We next put structure on $\ts{Fact}~A$. For a binary decomposition $f:A\to B\times C$, there is a related decomposition $f':A\to C\times B$, and this naturally defines a unary operation $'$ on $\ts{Fact}~A$. For a ternary decomposition $h:A\to B\times C\times D$, there are binary decompositions $h_{1}:A\to B\times (C\times D)$ and $h_{2}:A\to (B\times C)\times D$ given in an obvious way. We define a relation $\leq$ on $\ts{Fact}~A$ by setting one equivalence class of binary decompositions to be $\leq$ another if there is a ternary decomposition $h:A\to B\times C\times D$ with the first equivalence class containing $h_{1}:A\to B\times (C\times D)$ and the second $h_{2}:A\to (B\times C)\times D$. The following was established in \cite{Harding1}.

\begin{thm}
For $A$ a set, group, ring, vector space, topological space, or relational structure, $\ts{Fact}~A$ is an $\ts{omp}$. 
\end{thm}

The structure on $\ts{Fact}~X$ for a set $X$ can be described using factor pairs \cite{Harding1}. We choose to use the orthogonality relation $\perp$ as primitive rather than $\leq$ as it appears frequently in the sequel. Of course, these are interdefinable using the orthocomplementation $'$. We also make use of the fact that an \ts{omp} is isomorphic to its order dual to express the partial ordering and orthogonality in what seems a more natural way. 

\begin{thm}
When $\ts{Fact}~X$ is realized as factor pairs, we have $(\theta_1,\theta_2)'=(\theta_2,\theta_1)$ and $(\theta_1,\theta_2)\perp (\psi_1,\psi_2)$ iff $\theta_1\subseteq\psi_2$, $\psi_1\subseteq\theta_2$, and $\theta_1$ permutes with $\psi_1$.
\end{thm}

$\ts{Fact}~A$ has been studied in a number of papers. In \cite{Harding1}, the construction and its basic properties were introduced; it was shown that several familiar methods for constructing \ts{omp}s were instances of this construction; two related constructions of \ts{omp}s from relation algebras and symmetric lattices were given and used to show every modular ortholattice arises this way; and an example was given of a finite \ts{omp} that cannot be embedded into any $\ts{Fact}~A$. In \cite{Harding2} the blocks (maximal Boolean subalgebras) of $\ts{Fact}~A$ were described via Boolean sheaves, and the regularity of $\ts{Fact}~A$ established. In \cite{Harding3} direct physical motivation for the use of $\ts{Fact}~A$ in theoretical quantum mechanics was given. In \cite{Harding4} an example of a $\ts{Fact}~A$ with no states was given, and it was shown there is a finite \ts{Omp} that can be embedded into $\ts{Fact}~A$ for an infinite set $A$ but not for any finite set $A$. In \cite{Harding5} $\ts{Fact}~A$ was considered in a general categorical context. A survey of these results is given in \cite{Harding6}. In \cite{Harding7,Harding8} $\ts{Fact}~A$ is related to various categorical treatments of quantum mechanics and quantum logic \cite{BobSamson,HeunenJacobs}.  

The purpose of this paper is to study finer properties of the \ts{omp}s $\ts{Fact}~X$ for the case that $X$ is a finite set. We do this as a means of beginning a study of such properties for various kinds of structures due to connections to quantum logic, and as we believe the decompositions of a finite set form an object of basic interest, much as the partition lattice of a finite~set. In particular, we study properties of automorphisms. 

To frame the discussion, we state the following easily proved result that relates the automorphism group $\ts{Aut}(A)$ of the structure $A$, to the automorphism group $\ts{Aut}(\ts{Fact}~A)$ of the \ts{omp} $\ts{Fact}~A$. We note that an automorphism of an \ts{omp} is an order isomorphism that preserves the orthocomplementation. 

\begin{thm}
There is a group homomorphism $\Gamma:\ts{Aut}(A)\to\ts{Aut}(\ts{Fact}~A)$ where $\Gamma(\alpha)$ takes the equivalence class of the decomposition $f:A\to B\times C$ to the equivalence class of the decomposition $f\circ\alpha:A\to B\times C$. 
\label{a}
\end{thm}

There are several results known regarding automorphisms of structures $\ts{Fact}~A$. For a Hilbert space $\mathcal{H}$, the structure $\ts{Fact}~\mathcal{H}$ is the orthomodular lattice of closed subspaces \cite{Harding1}. Ulhorn's \cite{Ulhorn} version of Wigner's theorem shows that the automorphisms of $\ts{Fact}~\mathcal{H}$ are given by the unitary and anti-unitary operators on $\mathcal{H}$. In a series of papers by Chevalier and Ovchinnikov \cite{Chevalier1,Chevalier2,Chevalier3,Ovchinnikov} this was generalized to show the automorphisms of $\ts{Fact}~V$ for a vector space $V$ are given by isomorphisms and dual isomorphisms of the subspace lattice of $V$. In the finite-dimensional case, the fundamental theorem of projective geometry then provides a description of the automorphism group of $\ts{Fact}~V$ in terms of the general linear group of $V$. 

In this paper, the second section begins with some elementary combinatorial computations to find the number of atoms, number of blocks, size of blocks, and so forth, for the structures $\ts{Fact}~X$ for a finite set $X$ and $\ts{Fact}~V$ for a finite-dimensional vector space $V$. For sets with $p^k$ elements for some prime $p$, and for finite-dimensional vector spaces over finite fields, these are seen to give interesting classes of $(n,m)$-homogeneous \ts{omp}s \cite{Sultanbekov}, that is, ones where each block has $m$ atoms and each atom is in $n$ blocks. Several less basic combinatorial properties of these structures are also considered, such as the relationship between $\ts{Fact}~V$ when $V$ is considered as a vector space and $\ts{Fact}~V$ when $V$ is considered as a set. Many of the results of this section are used when we consider automorphisms. 

In the third section, we consider basic properties of automorphisms of $\ts{Fact}~X$ for a $X$ a finite set. We show that if $|X|$ is neither prime nor equal to 4, then the map $\Gamma$ is an embedding. Calling the kernel of $\Gamma$ the phase group of the structure, this says that except in some trivial cases, finite sets have trivial phase groups. We also show that the automorphism group of $\ts{Fact}~X$ is transitive on the atoms in a strong way, moving any block to any other. We then completely describe the automorphism group of $\ts{Fact}~X$ in the ``small'' cases where $|X|$ has at most two prime factors, and where $|X|=8$. 

Automorphisms groups in these ``small'' cases are somewhat uninteresting as the structures involved are too poor to allow control over automorphisms. This is particularly true when $|X|$ has just two prime factors as $\ts{Fact}~X$ is the horizontal sum of 4-element Boolean algebras. In the smallest case where $|X|$ has three prime factors, when $|X|=8$, $\ts{Fact}~X$ is again a horizontal sum, but of more complex pieces. This kind of pathology with small size is familiar in quantum logic with most results about Hilbert spaces having exceptions in the case of dimension 2. We conjecture below that these pathologies vanish as the size of the set becomes sufficiently large. 

\begin{conjecture}
If $X$ is a finite set with at least three prime factors greater than 2, then the map $\Gamma:\ts{Aut}(X)\to\ts{Aut}(\ts{Fact}~X)$ is an isomorphism. 
\end{conjecture}

In the fourth and fifth sections, we verify this conjecture in the case that $|X|=3^3$. This is not an easy task. The structure $\ts{Fact}~X$ has 5,001,134,190,558,105,600,000 atoms, and we are considering automorphisms of this structure. The proof proceeds in two steps. The first consists of showing each automorphism of $\ts{Fact}~X$ induces an automorphism on the poset of regular equivalence relations on $X$, and is the content of Section~4. The second step shows that each automorphism of the poset of regular equivalence relations on $X$ is given by a permutation of $X$, and is the content of Section~5. Both halves of the proof rely on on a sequence of elementary combinatorial computations. The method of proof may extend to the more general setting, but will require further non-trivial effort. 

In the sixth and final section, we make some remarks regarding group-valued states on the structures $\ts{Fact}~X$ and $\ts{Fact}~V$, and discuss directions for possible further research. 

\section{Counting with sets and vector spaces}

Here we employ some basic combinatorial techniques to describe properties such as the number of atoms in $\ts{Fact}~X$ for a finite set $X$. We begin with a generalization of the notion of a factor pair described in the introduction. We call a sequence of equivalence relations $(\theta_1,\ldots,\theta_n)$ on a set $X$ a factor $n$-tuple if the natural map $X\leadsto X/\theta_1\times\cdots\times X/\theta_n$ is a bijection. Factor $n$-tuples can be described concretely for arbitrary sets \cite{McKenzie}, but the finite sets there is a very simple alternative description. 

\begin{lemma}
Suppose $(\theta_1,\ldots,\theta_n)$ is an $n$-tuple of equivalence relations on a finite set $X$ where $\theta_i$ has $m_i$ blocks. Then this is a factor $n$-tuple iff 
\vspace{1ex}
\begin{enumerate}
\item $\theta_1\cap\cdots\cap\theta_n=\Delta$.
\item $|X| = m_1m_2\cdots m_n$. 
\end{enumerate}
\label{dddd}
\end{lemma}

\begin{proof} 
The first condition means the map $X\leadsto X/\theta_1\times\cdots\times X/\theta_n$ is a one-one map, and $\theta_i$ having $m_i$ blocks means $X/\theta_i$ has $m_i$ elements. 
\end{proof}

An equivalence relation that occurs as part of a factor pair, or equivalently as part of a factor $n$-tuple, is called a factor relation. Each factor relation on a set is regular, meaning that all its equivalence classes have the same cardinality, and each regular equivalence relation on a set is a factor relation. We call one factor relation a companion of another if the pair forms a factor pair. 

\begin{prop}
Suppose a finite set $X$ has $mn$ elements. Then there are \[{\displaystyle \frac{(mn)!}{m!(n!)^m}}\] factor relations with $m$ blocks of $n$ elements each, and each has $(n!)^{m-1}$ companions. 
\label{crap}
\end{prop}

\noindent {\bf Proof. } There are $(mn)!$ listings of the elements of $X$. The factor relations with $m$ blocks of $n$ elements each arise by taking the listings and inserting division lines after each batch of $n$ elements, separating the listing of $X$ into $m$ groups of $n$. The order of the groups does not matter, and the order of the elements within the groupings does not matter, so we divide by $m!$ and we divide by $(n!)^m$. This gives the formula for the number of factor relations with $m$ blocks of $n$ elements each. 

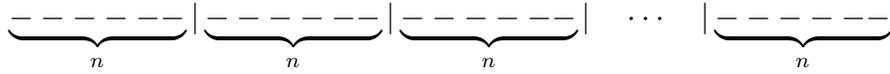
\[ \underbrace{ ------ }_{n} | \underbrace{ ------ }_{n} | \underbrace{ ------ }_{n} | \quad \cdots \quad  | \underbrace{ ------ }_{n} \]
\vspace{1ex}

Suppose $\theta$ is a factor relation with $m$ blocks of $n$ elements each and $\phi$ is a companion of $\theta$. Then $\phi$ has $n$ blocks of $m$ elements each and each block of $\phi$ has exactly one element of each block of $\theta$. We show $\theta$ below with the blocks as rows. 

\[\begin{array}{ccccccc}
a_{11}&a_{12}&a_{13}&&\cdots&&a_{1n}\\
a_{21}&a_{22}&a_{23}&&\cdots&&a_{2n}\\
\vdots&\vdots&\vdots&&\vdots&&\vdots \\
a_{m1}&a_{m2}&a_{m3}&&\cdots&&a_{mn}
\end{array}
\]
\vspace{1ex}

\noindent One block of $\phi$, we call it the first, will contain $a_{11}$, one block, call it the second, will contain $a_{12}$, and so forth. To fill out the rest of the first block of $\phi$ we must choose one element from each row of $\theta$ after the first row, so there are $n^{m-1}$ ways to choose the first block of $\phi$. Once the first block of $\phi$ is chosen, we choose the second block by choosing one element from each of rows $2,\ldots,m$ of $\theta$ not already chosen for the first block of $\phi$. So there are $(n-1)^{m-1}$ ways to choose the second block, and so forth. In total, we have $(n)^{m-1}(n-1)^{m-1}(n-2)^{m-1}\cdots (2)^{m-1}(1)^{m-1}$
ways to select $\phi$. But this simplifies to $(n!)^{m-1}$ ways to choose $\phi$. $\Box$

\begin{coro}
If $|X|=mn$ then the number of factor pairs $(\theta_1,\theta_2)$ where $\theta_1$ has $m$ blocks of $n$ elements each and $\theta_2$ has $n$ blocks of $m$ elements each is given by $${\displaystyle \frac{(mn)!}{m!n!}}.$$
\label{count1}
\end{coro}

This lets us count the number of atoms in $\ts{Fact}~X$ for any finite set. In general, the atoms will be of the form $(\theta_1,\theta_2)$ where the blocks of $\theta_1$ have a prime number of elements. Usually the atoms of $\ts{Fact}~X$ will come in different flavors depending on the different primes that divide $|X|$. We can use the above to count the number of each flavor, but we keep it simple and limit ourselves to the following. 

\begin{coro}
If $|X|=p^k$ with $p$ prime, then the number of atoms in $\ts{Fact}~X$ is 
$${\displaystyle \frac{p^k!}{p!(p^{k-1})!}}.$$
\label{numatoms}
\end{coro}

We next turn to counting the blocks (maximal Boolean subalgebras) of $\ts{Fact}~X$. We again restrict attention to the case where $|X|=p^k$ is a prime power, and call attention to the fact that we use the term block both for an equivalence class of an equivalence relation, and for a maximal Boolean subalgebra of an \ts{omp}. The key is the following result established in \cite{Harding2}. 

\begin{coro}
If $X$ is a finite set, then the blocks of $\ts{Fact}~X$ correspond to unordered versions of factor $n$-tuples $(\theta_1,\ldots,\theta_n)$ where each $\theta_i$ has a prime number of blocks. 
\label{po}
\end{coro}

Note, factor $n$-tuples correspond to blocks with a specific order to their atoms. Permuting a factor $n$-tuple gives a new factor $n$-tuple, but yields the same block. 

\begin{lemma}
If $|X|=p^k$ where $p$ is prime, then the number of blocks in $\ts{Fact}~X$ is 
$$\frac{p^k!}{k!(p!)^k}.$$
\label{numblocks}
\end{lemma}

\noindent {\bf Proof. } By Corollary~\ref{po} we want the number of factor $k$-tuples $(\theta_1,\ldots,\theta_k)$ where each $\theta_i$ has $p$ blocks with $p^{k-1}$ elements each, divided by $k!$. For such a factor $k$-tuple, $\theta_1\cap\theta_2$ has $p^2$ blocks of $p^{k-2}$ elements each, $\theta_1\cap\theta_2\cap\theta_3$ has $p^3$ blocks of $p^{k-3}$ elements each, and so forth. We build $\theta_1,\theta_2,\ldots,\theta_k$ with this in mind. By Proposition~\ref{crap}, the number of ways to choose $\theta_1$ is given by

$${\displaystyle \frac{(p^k)!}{p!((p^{k-1})!)^{p-1}}}.$$
\vspace{1ex}

With $\theta_1$ is chosen, we begin to construct $\theta_2$. To choose the first block of $\theta_2$ we choose $p^{k-2}$ elements from the $p^{k-1}$ elements of the first block of $\theta_1$, $p^{k-2}$ elements from the $p^{k-1}$ elements of the second block of $\theta_1$, and so forth. Using standard notation $(\begin{smallmatrix}m\\n\end{smallmatrix})$ for $m$ choose $n$, we may choose the first block of $\theta_2$ in the following number of ways.
\[\left(\begin{array}{c}p^{k-1}\\p^{k-2}\end{array}\right)^p\]
\vspace{1ex}

Building the second block of $\theta_2$ is similar, but from each of the $p$ blocks of $\theta_1$ we choose $p^{k-2}$ of the $p^{k-1}-p^{k-2}=(p-1)p^{k-2}$ elements in that block not already chosen. Building the third block of $\theta_2$ is also similar, but from each of the $p$ blocks of $\theta_1$ we choose $p^{k-2}$ of the $p^{k-1}-2p^{k-2}=(p-2)p^{k-2}$ elements not already chosen. As the order of the $p$ blocks of $\theta_2$ does not matter (dividing by $p!$) we may choose $\theta_2$ in the following number of ways. 

\[
\frac{1}{p!}
\left[
\left(\!\!\begin{array}{c}pp^{k-2}\\p^{k-2}\end{array}\!\!\right)\!
\left(\!\!\begin{array}{c}(p-1)p^{k-2}\\p^{k-2}\end{array}\!\!\right)\!
\left(\!\!\begin{array}{c}(p-2)p^{k-2}\\p^{k-2}\end{array}\!\!\right) \,\,
\cdots\cdots \,\,
\left(\!\!\begin{array}{c}p^{k-2}\\p^{k-2}\end{array}\!\!\right)
\right]^p
\]
\vspace{1ex}

\noindent Setting $u=p^{k-2}$ to aid legibility, this expression is equal to 

\[ 
\frac{1}{p!}
\left[
\frac{p^{k-1}!}{((p-1)u)!u!}
\frac{((p-1)u)!}{((p-2)u)!u!}
\frac{((p-2)u)!}{((p-3)u)!u!}\,\,
\cdots\,\,
\frac{(3u)!}{(2u)!u!}
\frac{(2u)!}{u!u!}
\frac{u!}{0!u!}
\right]^p
\]
\vspace{1ex}

\noindent After simplification, given $\theta_1$, the number of ways to choose $\theta_2$ equals

\[\frac{1}{p!}\left[ \frac{p^{k-1}!}{(p^{k-2}!)^p} \right]^p\]
\vspace{1ex}

Suppose $\theta_1$ and $\theta_2$ are chosen. Then $\theta_1\cap\theta_2$ has $p^2$ blocks with $p^{k-2}$ elements each. To construct the first block of $\theta_3$ choose $p^{k-3}$ elements from the $p^{k-2}$ elements of each of these $p^2$ blocks of $\theta_1\cap\theta_2$. So there will be $(\begin{smallmatrix}p^{k-2}\\p^{k-3}\end{smallmatrix})^{p^2}$ ways to select the first block. Setting $v=p^{k-3}$ and proceeding as above, the number of ways to construct $\theta_3$ is given by 

\[ 
\frac{1}{p!}
\left[
\frac{p^{k-2}!}{((p-1)v)!v!}
\frac{((p-1)v)!}{((p-2)v)!v!}
\frac{((p-2)v)!}{((p-3)v)!v!}\,\,
\cdots\,\,
\frac{(3v)!}{(2v)!v!}
\frac{(2v)!}{v!v!}
\frac{v!}{0!v!}
\right]^{p^2}
\]
\vspace{1ex}

\noindent So given $\theta_1,\theta_2$, the number of ways to choose $\theta_3$ is 

\[\frac{1}{p!}\left[ \frac{p^{k-2}!}{(p^{k-3}!)^p} \right]^{p^2}\]
\vspace{1ex}

Proceeding, the number of ways to choose $\theta_1,\theta_2,\ldots,\theta_k$ is 

\[
\frac{1}{p!}\left[\frac{p^k!}{(p^{k-1}!)^p}\right]^1
\frac{1}{p!}\left[\frac{p^{k-1}!}{(p^{k-2}!)^p}\right]^p
\frac{1}{p!}\left[\frac{p^{k-2}!}{(p^{k-3}!)^p}\right]^{p^2}
\cdots\,\,\,
\frac{1}{p!}\left[\frac{p^2!}{(p!)^p}\right]^{p^{k-2}}
\frac{1}{p!}\left[\frac{p!}{(1!)^p}\right]^{p^{k-1}}
\]
\vspace{1ex}

\noindent This simplifies to 

\[
\frac{1}{(p!)^k}\frac{p^k!}{(p^{k-1}!)^p}
\frac{(p^{k-1}!)^p}{(p^{k-2}!)^{p^2}}
\frac{(p^{k-2}!)^{p^2}}{(p^{k-3}!)^{p^3}}
\,\,\cdots\,\,
\frac{(p^3!)^{p^{k-3}}}{(p^2!)^{p^{k-2}}}
\frac{(p^2!)^{p^{k-2}}}{(p!)^{p^{k-1}}}
\frac{(p!)^{p^{k-1}}}{(1!)^{p^k}}
\]
\vspace{1ex}

\noindent This is the number of ways to choose $\theta_1,\cdots,\theta_k$ where order matters. Dividing this by $k!$ gives the number of blocks, and the result after noticing that many terms above cancel, is the following. 
\[ \frac{p^k!}{k!(p!)^k}.\]

\noindent This is the desired formula. $\Box$

\begin{prop}
Suppose $|X|=p^k$ with $p$ prime. Then each block of $\ts{Fact}~X$ has $k$ atoms, and each atom belongs to the following number of blocks. 

\[\frac{p^{k-1}!}{(k-1)!(p!)^{k-1}}\]
\label{fhf}
\end{prop}
\vspace{1ex}

\noindent {\bf Proof. } As blocks of $\ts{Fact}~X$ are given by factor $k$-tuples $(\theta_1,\ldots,\theta_k)$ where each $\theta_i$ has $p$ blocks, it follows from results in \cite{Harding2} that each block of $\ts{Fact}~X$ has $k$ atoms. Let $A$ be the number of atoms and $B$ the number of blocks in $\ts{Fact}~X$. Above we have seen 
\[ A = \frac{p^k!}{p!(p^{k-1})!} \quad\mbox{ and }\quad B = \frac{p^k!}{k!(p!)^k}.\]
\vspace{1ex}

\noindent As each block has $k$ atoms in it, the average number of blocks an atom belongs to is 

\[ \frac{kB}{A}=\frac{p^{k-1}!}{(k-1)!(p!)^{k-1}}\]
\vspace{1ex}

\noindent Later, in Proposition~\ref{gnn}, we will see that if $|X|$ is a prime power, then there is an automorphism of $\ts{Fact}~X$ taking any given atom to any other. It follows that all atoms are in the same number of blocks, so this average is realized by all atoms. $\Box$

\begin{rmk}{\em An \ts{omp} is called $(n,m)$-homogeneous \cite{Sultanbekov} if all of its blocks have $m$ atoms and each atom is in $n$ blocks. The above results show that if $X$ is a set with $p^k$ elements for some prime $p$, then $\ts{Fact}~X$ is $(n,m)$-homogeneous where $m=k$ and $n$ is given in Proposition~\ref{fhf}. In the next section we will see that this homogeneity arises in a very strong way, from the fact that given any two blocks with any sequencing of their atoms, there is an automorphism of $\ts{Fact}~X$ taking the atoms of one block to the atoms of the other that respects the sequencing of these atoms. We call such an \ts{omp} strongly transitive. Any strongly transitive \ts{omp} is $(n,m)$-homogeneous, but not conversely, as is seen by taking the horizontal sum of two non-isomorphic $(n,m)$-homogeneous \ts{omp}s. 
}
\end{rmk}

Factor pairs sharing a component, such as $(\alpha,\gamma)$ and $(\beta,\gamma)$, play an important role in the study of automorphisms, and also in the structure of $\ts{Fact}~A$. The following result, with a very pleasant proof, clarifies their situation. 

\begin{prop}
If $\alpha,\beta$ are equivalence relations on a set $X$, with each having the same finite number $k$ of blocks all with the same cardinality, then there is an equivalence relation $\gamma$ so that both $(\alpha,\gamma)$ and $(\beta,\gamma)$ are factor pairs. 
\end{prop}

\begin{proof}
We first consider the case where $X$ is finite. Then $\alpha$ partitions $X$ into $k$ pieces $X_1,\ldots,X_k$ with each piece having $n$ elements, and $\beta$ also partitions $X$ into $k$ pieces $Y_1,\ldots,Y_k$, again with each piece having $n$ elements. We will show there is a partition of $X$ into $n$ pieces $Z_1,\ldots,Z_n$ with each piece having $k$ elements so that $|X_i\cap Z_j|=|Y_i\cap Z_j| = 1$ for each $i\leq k$ and $j\leq n$. This will establish the result in the case that $X$ is finite. 

\begin{claim} 
There is $Z_1=\{x_1,\ldots,x_k\}$ with $|X_i\cap Z_1| = |Y_i\cap Z_1| = 1$ for each $i\leq k$.
\end{claim} 

\begin{proof}[Proof of Claim]  
Let $S_p=\{i:X_p\cap Y_i\neq\emptyset\}$ for $p=1,\ldots,k$. Let $\mathcal{S}=\{S_1,\ldots,S_k\}$. Take an arbitrary number $m$ of these sets $S_i$, without loss of generality, $S_1,\ldots,S_m$. We claim their union has at least $m$ elements. As $X$ is covered by the $Y_i$'s, it follows that the mn-element set $X_1\cup\cdots\cup X_m$ is covered by the $Y_i$'s and hence by the $Y_i$'s where $i\in S_1\cup\cdots\cup S_m$. As each $Y_i$ has $n$ elements, there must be at least $m$ of the $i$'s belonging to $S_1\cup\cdots\cup S_m$. So $\mathcal{S}$ satisfies the conditions for Phillip Hall's Marriage Theorem, so has a system of distinct representatives. This means we can choose distinct $j_1,\ldots,j_k$ so that $j_i\in S_i$ for each $i\leq k$. This means $X_i\cap Y_{j_i}\neq\emptyset$ for each $i\leq k$. So we can pick $x_i\in X_i\cap Y_{j_i}$. Set $Z_1=\{x_1,\ldots,x_k\}$. 
\end{proof}

\begin{claim} 
There is a partition $Z_1,\ldots,Z_n$ of $X$ with each $Z_i$ having $k$ elements, and such that $|X_i\cap Z_j|=|Y_i\cap Z_j|=1$ for each $i\leq k$ and $j\leq n$. 
\end{claim}

\begin{proof}[Proof of Claim]  
By induction on $n$ for fixed $k$. Claim~1 shows we can find $Z_1$ with $k$ elements so that $Z_1$ hits each $X_i$ exactly once and each $Y_i$ exactly once. Let $X'=X-Z_1$, $X_i'=X_i-Z_1$ and $Y_i'=Y_i-Z_1$. Then $X_1',\ldots,X_k'$ and $Y_1',\ldots,Y_k'$ are partitions of $X'$ into $(n-1)$-element pieces. By the inductive hypothesis we can find $Z_2,\ldots,Z_n$ partitioning $X'$ into $k$-element pieces so that $|X_i'\cap Z_j|=|Y_i'\cap Z_j|=1$ for each $i\leq k$ and $2\leq j\leq n$. Then $Z_1,\ldots,Z_n$ has the desired properties. 
\end{proof}

Having established the result for the finite case, the proof of the infinite case is identical using the infinite version of Phillip Hall's theorem (proved by the unrelated Marshall Hall). This result requires that we have finitely many infinite sets, which is the case as we have required $\alpha,\beta$ to have finitely many blocks.  
\end{proof}

We next turn our attention to $\ts{Fact}~V$ for a vector space $V$. Results of \cite{Harding2} show that if $V$ is finite-dimensional, of dimension $k$, then the blocks of $\ts{Fact}~V$ each have $k$ atoms. If $V$ is finite-dimensional and over a finite field, then $V$ itself is finite, and we may employ counting techniques similar to those used for sets with $\ts{Fact}~V$. This will be our primary interest here. We begin with our basic representation for vector spaces, established in \cite{Harding1}. 

\begin{thm}
For $V$ a vector space, $\ts{Fact}~V$ may be realized as the set of all ordered pairs of complementary subspaces $(S,T)$ of $V$, meaning pairs with $S\cap T=\{0\}$ and $S+T=V$, where $(S,T)'=(T,S)$ and $(S_1,T_1)\leq (S_2,T_2)$ iff $S_1\subseteq S_2$ and $T_2\subseteq T_1$. 
\end{thm}

$\ts{Fact}~V$ for a vector space $V$ may be viewed as a special case of other more general methods to construct \ts{omp}s, such as the \ts{omp} of idempotents of a ring, or as the \ts{omp} of certain complementary pairs of elements of a symmetric lattice \cite{Harding1}. 

\begin{lemma}
Let $V$ be a $k$-dimensional vector space over an $n$-element field. Then 
\begin{enumerate}
\item {\em Fact}~$V$ has ${\displaystyle \frac{n^k-1}{n-1}n^{k-1}}$ atoms.
\item[]
\item {\em Fact}~$V$ has ${\displaystyle \frac{(n^k-1)(n^k-n)(n^k-n^2)\,\cdots\,(n^k-n^{k-1})}{k!\,(n-1)^k}}$ blocks.
\item[]
\item Each atom belongs to ${\displaystyle \frac{(n^k-n)(n^k-n^2)(n^k-n^3)\,\cdots\,(n^k-n^{k-1})}{(k-1)!\,n^{k-1}\,(n-1)^{k-1}}}$ blocks.  
\end{enumerate}
\label{countFactV}
\end{lemma}

\begin{proof} There are $n^k-1$ non-zero elements in $V$. Each non-zero vector has $n-1$ non-zero scalar multiples, so each one-dimensional subspace has $n-1$ non-zero elements. Thus
\begin{equation}
\mbox{The number of 1-dim subspaces is: }\quad\frac{n^k-1}{n-1}
\end{equation}
\vspace{1ex}

\noindent By a general argument using annihilators, there are the same number of $d$-dimensional subspaces as $(k-d)$-dimensional subspaces, and this gives the following. 

\begin{equation}
\mbox{The number of (k-1)-dim subspaces is: }\quad\frac{n^k-1}{n-1}
\end{equation}
\vspace{1ex}

\noindent Now each $(k-1)$-dimensional subspace has $n^{k-1}-1$ non-zero vectors in it. So there are $(n^k-1)-(n^{k-1}-1)=n^k-n^{k-1}$ non-zero vectors not in it. Each 1-dimensional subspace disjoint from a $(k-1)$-dimensional subspace has $n-1$ non-zero vectors in it, so each $(k-1)$-dimensional subspace has $(n^k-n^{k-1})/(n-1)=n^{k-1}$ 1-dimensional subspaces disjoint from it. Thus 

\begin{equation}
\mbox{The number of atoms is: }\quad\frac{n^k-1}{n-1}n^{k-1}
\end{equation}
\vspace{1ex}

From general considerations from \cite{Harding2}, blocks of Fact~$V$ correspond to sets of $k$ 1-dimensional subspaces whose collective span is all of $V$. This means unordered bases, up to scalar multiples of the basis elements. There are $(n^k-1)/(n-1)$ ways to choose the first basis element. Once chosen, we need a non-zero vector not in this space, so there are $[(n^k-1)-(n-1)]/(n-1)=(n^k-n)/(n-1)$ ways to do this up to equivalence by scalar multiples. For the third, we need a non-zero vector not in the 2-dimensional span of what we have so far. There are $[(n^k-1)-(n^2-1)]/(n-1)=(n^k-n^2)/(n-1)$ ways to do this up to scalar multiples. In total there are the following number of ways to choose this ordered basis up to equivalence of scalar multiples. 

\begin{equation}
\frac{n^k-1}{n-1}\,\,\frac{n^k-n}{n-1}\,\,\frac{n^{k}-n^2}{n-1}\,\,\frac{n^{k}-n^3}{n-1}\,\,\cdots\,\,\frac{n^k-n^{k-1}}{n-1}
\end{equation}
\vspace{1ex}

\noindent So dividing out the order we then have 

\begin{equation}
\mbox{The number of blocks is: }\quad\frac{(n^k-1)(n^k-n)(n^k-n^2)\,\cdots\,(n^k-n^{k-1})}{k!\,(n-1)^k}
\end{equation}
\vspace{1ex}

\noindent The average number of blocks each atom is in equals (the number of blocks) times (the number of atoms per block) divided by (the number of atoms), and there are $k$ atoms in each block. So 

\begin{equation}
\mbox{No. blocks each atom is in: }\quad
\frac{(n^k-n)(n^k-n^2)(n^k-n^3)\,\cdots\,(n^k-n^{k-1})}{(k-1)!\,n^{k-1}\,(n-1)^{k-1}}
\end{equation}
\vspace{1ex}

\noindent Here we use a result from the following section that the automorphism group of $\ts{Fact}~V$ is transitive on atoms. So the average number of blocks each atom is in is attained by each atom.
\end{proof}

We give several computations in specific cases. 

\begin{examp}{\em 
For $V=\Bbb{Z}_2^3$, $\ts{Fact}~V$ has 28 atoms, 28 blocks, each block has 3 atoms, and each atom is in 3 blocks. For $V=\Bbb{Z}_3^3$, $\ts{Fact}~V$ has 117 atoms, 234 blocks, each block has 3 atoms, and each atom is in 6 blocks. For $X$ an 8-element set, $\ts{Fact}~X$ has 840 atoms, 840 blocks, each block has 3 atoms and each atom is in 3 blocks. For $X$ a 27-element set, things become very large. $\ts{Fact}~X$ has $27!/3!9!= 5,001,134,190,558,105,600,000$ atoms, $27!/3!6^3$ blocks, each block has 3 atoms, and each atom is in 10,080 blocks. 
}
\end{examp}

These counting arguments provide key insight into $\ts{Fact}~X$ for an 8-element set. We recall that a horizontal sum of a family of \ts{omp}s is obtained by taking their disjoint union and identifying their bottom and top elements $0$ and $1$. 

\begin{prop}
For $X$ an 8-element set, the \ts{omp} $\ts{Fact}~X$ is a horizontal sum of 30 copies of the \ts{omp} $\ts{Fact}~\Bbb{Z}_2^3$. 
\label{8ishorizontal}
\end{prop}

\begin{proof}
Suppose we define an addition $+$ and zero $0$ on $X$ so that that the resulting structure $V=(X,+,0)$ is a 3-dimensional vector space over $\Bbb{Z}_2$. Let $\ts{Fact}~V$ be the set of all factor pairs $(\theta_1,\theta_2)$ where $\theta_1$ and $\theta_2$ are congruences with respect to this vector space structure. Then $\ts{Fact}~V$ is a sub-$\ts{omp}$ of $\ts{Fact}~X$ and $\ts{Fact}~V_+$ is isomorphic to the \ts{omp} $\ts{Fact}~\Bbb{Z}_2^3$. 

We claim that $\ts{Fact}~X$ is the horizontal sum of its subalgebras that arise as $\ts{Fact}~V$ for some vector space structure $V$ on $X$. To establish the claim, it is enough to show that each block of $\ts{Fact}~X$ is contained in some subalgebra $\ts{Fact}~V$, and that any block of $\ts{Fact}~X$ that contains an atom of some $\ts{Fact}~V$ is contained in this subalgebra. Once this claim is established, it follows from the above counting arguments that there are 30 such horizontal summands. 

For any block of $\ts{Fact}~X$, there is a corresponding factor triple $(\theta_1,\theta_2,\theta_3)$. Each $X/\theta_i$ is a 2-element set, and choosing some way to put $\Bbb{Z}_2$-vector space structure on each of these 2-element sets and taking the product structure on $X$, we have $(\theta_1,\theta_2,\theta_3)$ is a factor triple of $V$. Then the given block is a block of the subalgebra $\ts{Fact}~V$. Suppose some block $B$ of $\ts{Fact}~X$ has an atom $a$ belonging to some $\ts{Fact}~V$. The above counting arguments show there are 3 blocks of $\ts{Fact}~X$ that contain $a$ and 3 blocks of $\ts{Fact}~V$ that contain $a$. So the block $B$ must be a block of $\ts{Fact}~V$. 
\end{proof}

\begin{rmk}{\em 
For any set $X$ whose cardinality is a prime power $p^k$, we may consider subalgebras of $\ts{Fact}~X$ of the form $\ts{Fact}~V$ where $V$ is some $k$-dimensional $\Bbb{Z}_p$-vector space structure on $X$. We call such subalgebras $\Bbb{Z}_p$-blocks. One can show 
\vspace{2ex}
\begin{enumerate}
\item $\ts{Fact}~X$ has $\,{\displaystyle \frac{(p^3-1)!}{(p^3-1)(p^3-p)(p^3-p^2)}}\,$ $\Bbb{Z}_p$-blocks. 
\vspace{3ex}
\item Each atom of $\ts{Fact}~X$ is in ${\displaystyle \frac{(p^2-1)!(p-2)!}{(p^2-1)(p^2-p)}}$ $\Bbb{Z}_p$-blocks. 
\vspace{3ex}
\item Each block of $\ts{Fact}~X$ is in $(p-2)!$ $\Bbb{Z}_p$-blocks.
\end{enumerate}
\vspace{2ex}
\noindent When $|X|=27$, each block of $\ts{Fact}~X$ is in just one $\Bbb{Z}_3$-block and each atom is in 840 $\Bbb{Z}_3$-blocks. So $\ts{Fact}~X$ is no longer a horizontal sum as in the case of $|X|=8$.}
\end{rmk}


\section{Automorphisms}

In this section, we make some general remarks about the group homomorphism $\Gamma:\ts{Aut}(A)\to \ts{Aut}(\ts{Fact}~A)$ in the case that $A$ is a finite set, or a finite-dimensional vector space. We begin with a description of $\Gamma$ when applied to $\ts{Fact}~X$ for a set $X$. 

\begin{defn}
For a set $X$, and permutation $\alpha$ of $X$, define for each relation $\theta$ on $X$ the relation $\alpha\theta=\{(\alpha x,\alpha y):(x,y)\in\theta\}$. Then define $\Gamma:\ts{Aut}(X)\to\ts{Aut}(\ts{Fact}~\!X)$ by setting $\Gamma \alpha$ to be the map with $(\Gamma\alpha)(\theta,\theta')=(\alpha\theta,\alpha\theta')$. 
\end{defn}

\begin{prop}
If $X$ is a set whose cardinality is neither prime, nor equal to 4, then the map $\Gamma$ is an embedding. Consequently, the phase group of $X$ is trivial. 
\label{phasegroup}
\label{b}
\end{prop}

\begin{proof}
In these circumstances, there are $m,n$ with $m\geq 2$ and $n\geq 3$ with $|X|=mn$. This includes the infinite case if we allow $n$ to be infinite, but our diagrams will indicate the finite case. Suppose $\alpha$ is a permutation of $X$ with $\Gamma\alpha$ the identity on $\ts{Fact}~X$. Suppose also there is some $a\in X$ with $\alpha a\neq a$. Then we can enumerate the elements of $X$ as $a_{ij}$ where $i\leq m$ and $j\leq n$ so that $a=a_{11}$ and $\alpha (a)= a_{12}$. 

\[\begin{array}{ccccccc}
a_{11}&a_{12}&a_{13}&&\cdots&&a_{1n}\\
a_{21}&a_{22}&a_{23}&&\cdots&&a_{2n}\\
\vdots&\vdots&\vdots&&\vdots&&\vdots \\
a_{m1}&a_{m2}&a_{m3}&&\cdots&&a_{mn}
\end{array}
\]
\vspace{2ex}

There is a factor pair $(\theta_1,\theta_2)$ with blocks of $\theta_1$ being rows in the above diagram and blocks of $\theta_2$ being the columns. As $\Gamma\alpha$ is the identity, $\alpha\theta_1=\theta_1$ and $\alpha\theta_2=\theta_2$. So $\alpha$ maps the elements of one row to those of another, and the elements of one column to those of another. Thus there are permutations $\mu$ of $\{1,\ldots,m\}$ and $\nu$ of $\{1,\ldots,n\}$ with $\alpha (a_{ij})=a_{\mu(i)\nu(j)}$ for each $i,j$.  As $\alpha(a_{11})=a_{12}$ we have $\mu(1)=1$ and $\nu(1)=2$. 

Swap places of $a_{11}$ and $a_{13}$ and consider next the factor pair $(\phi_1,\phi_2)$ where the blocks of $\phi_1$ are the rows of the diagram below, and the blocks of $\phi_2$ are the columns. 

\[\begin{array}{ccccccc}
a_{13}&a_{12}&a_{11}&&\cdots&&a_{1n}\\
a_{21}&a_{22}&a_{23}&&\cdots&&a_{2n}\\
\vdots&\vdots&\vdots&&\vdots&&\vdots \\
a_{m1}&a_{m2}&a_{m3}&&\cdots&&a_{mn}
\end{array}
\]
\vspace{1ex}

\noindent Then $\alpha(a_{12})$ belongs to the second column and $\alpha(a_{13})$ does not. So $\alpha\phi_2\neq\phi_2$, and this contradicts $\Gamma\alpha$ being the identity. 
\end{proof}

Let $A$ be a structure with all blocks of $\ts{Fact}~A$ finite. We say the automorphism group of $\ts{Fact}~\!A$ is transitive on atoms if for any two atoms of $\ts{Fact}~A$ there is an automorphism of $\ts{Fact}~A$ mapping the first to the second. We say the automorphism group is transitive on blocks if for any two blocks of $\ts{Fact}~A$ there is an automorphism of $\ts{Fact}~A$ mapping the first block to the second. We say the automorphism group is strongly transitive on blocks if for any two blocks, and any two sequencings of the atoms of these blocks, there is an automorphism taking the first block to the second and compatible with the given sequencings. It is easily seen that strong transitivity on blocks implies transitivity on blocks, and this implies transitivity on atoms. 

\begin{prop}
For a finite set $X$, the automorphism group of $\ts{Fact}~X$ is transitive on blocks. If $|X|$ is a prime power, then it is strongly transitive on blocks.  
\label{gnn}
\end{prop}

\begin{proof}
Here it is easiest to work with the definition of $\Gamma$ from the introduction, where a permutation $\alpha$ of $X$ is taken to the automorphism $\Gamma\alpha$ of $\ts{Fact}~X$ that maps the equivalence class of the decomposition $f:X\to Y\times Z$ to the equivalence class of the decomposition $f\circ\alpha:X\to Y\times Z$. 

Given two blocks of $\ts{Fact}~X$, there are decompositions $f:X\to Y_1\times\cdots\times Y_m$ and $g:X\to Z_1\times\cdots\times Z_n$, with each $Y_i$ and $Z_j$ directly irreducible, from which these blocks are built \cite{Harding2}. In particular, the atoms of these blocks are of the equivalence classes of the binary decompositions 

\[f_i:X\to Y_i\times(\prod_{j\neq i}Y_j)\quad\mbox{ and }\quad g_i:X\to Z_i\times(\prod_{j\neq i}Z_j)\]

As each $Y_i$ and $Z_j$ are directly irreducible, they have prime cardinality, and the cardinality of $X$ is the product of the cardinalities of the $Y_i$'s and of the $Z_j$'s. So $m=n$, and there exists a permuatation $\sigma$ of $n$ and bijections $h_i:Y_i\to Z_{\sigma(i)}$. 

\setlength{\unitlength}{.035in}
\begin{center}
\begin{picture}(80,35)(0,-5)
\put(0,0){\makebox(0,0)[r]{$X$}}
\put(0,20){\makebox(0,0)[r]{$X$}}
\put(5,0){\vector(1,0){30}}
\put(5,20){\vector(1,0){30}}
\put(-2.5,5){\vector(0,1){10}}
\put(40,0){\makebox(0,0)[l]{$Z_{\sigma 1}\,\times\,\cdots\,\times\, Z_{\sigma n}$}}
\put(42,20){\makebox(0,0)[l]{$Y_1\,\times\,\,\cdots\,\,\times\, Y_n$}}
\put(44,15){\vector(0,-1){10}}
\put(70,15){\vector(0,-1){10}}
\put(-5,10){\makebox(0,0)[r]{$\alpha$}}
\put(20,22){\makebox(0,0)[b]{$f$}}
\put(20,-3){\makebox(0,0)[t]{$g_\sigma$}}
\put(-5,10){\makebox(0,0)[r]{$\alpha$}}
\put(42,10){\makebox(0,0)[r]{$h_1$}}
\put(72,10){\makebox(0,0)[l]{$h_n$}}
\end{picture}
\end{center}
\vspace{2ex}

Denote $g(x)$ by $(g_1(x),\ldots,g_n(x))$ and define $g_\sigma(x)=(g_{\sigma 1}(x),\ldots,g_{\sigma n}(x))$. Define $k$ from $Z_1\times\cdots\times Z_n\to Y_1\times\cdots\times Y_n$ by setting $k(z_1,\ldots,z_n)=(h_1^{-1}(z_{\sigma 1}),\ldots,h_n^{-1}(z_{\sigma n}))$, and then let $\alpha$ be the permutation of $X$ given by $\alpha=f^{-1}\circ k\circ g$. Then for the usual product map $h=h_1\times\cdots\times h_n$, we have the following for each $x\in X$. 

\begin{eqnarray*}
(h\circ f\circ\alpha)(x)&=&(h\circ f\circ f^{-1}\circ k\circ g)(x)\\
&=&(h\circ k)(g_1(x),\ldots,g_n(x))\\
&=&h(h_1^{-1}(g_{\sigma 1}(x)),\ldots,h_n^{-1}(g_{\sigma n}(x)))\\
&=&(h_1h_1^{-1}(g_{\sigma 1}(x)),\ldots,h_nh_n^{-1}(g_{\sigma n}(x)))\\
&=&g_\sigma(x)
\end{eqnarray*}
\vspace{.5ex}

Extending the definition of $\Gamma\alpha$ to $n$-ary decompositions in the obvious way, we have that $\Gamma\alpha$ takes the equivalence class of the decomposition $f:X\to Y_1\times\cdots\times Y_n$ to the equivalence class of the decomposition $f\circ\alpha:X\to Y_1\times\cdots\times Y_n$. Then $h_1,\ldots,h_n$ are bijections showing that the $n$-ary decomposition $f\circ\alpha:X\to Y_1\times\cdots\times Y_n$ is equivalent to the $n$-ary decomposition $g_\sigma:X\to Z_{\sigma 1}\times\cdots\times Z_{\sigma n}$. So $\Gamma\alpha$ takes the equivalence class of $f:X\to Y_1\times\cdots\times Y_n$ to that of $g_{\sigma}:X\to Z_{\sigma 1}\times\cdots\times Z_{\sigma n}$. 

Consider the following binary decompositions. 

\[ f_i:X\to Y_i\times(\prod_{j\neq i}Y_j)\quad\mbox{ and }\quad (g_{\sigma})_i:X\to Z_{\sigma(i)}\times(\prod_{j\neq\sigma(i)}Z_j)\]
\vspace{1ex}

\noindent From the above remarks about $\Gamma\alpha$ and its action with respect to $f$ and $g_\sigma$, it follows that $\Gamma\alpha$ maps the equivalence class of the first binary decomposition to the equivalence class of the second. Thus $\Gamma\alpha$ maps the atoms of the first block to those of the second. Thus $\Gamma$ is transitive on the blocks. 

However, the sequencing of the matching of the atoms is determined by the permutation $\sigma$. If $|X|$ is a prime power, then all irreducible factors have the same prime number of elements, and the permutation $\sigma$ can be chosen to be the identity. In this case, $\Gamma$ is strongly transitive on the atoms. 
\end{proof}

The proof of the following result is nearly identical to the above. 

\begin{thm}
If $V$ is a finite-dimensional vector space, then the automorphism group of $\ts{Fact}~V$ is strongly transitive on blocks. 
\end{thm}

\begin{rmk}{\em 
In the setting of a structure $A$ where the blocks of $\ts{Fact}~A$ are all finite, the above proof shows that the automorphism group of $\ts{Fact}~A$ is transitive on blocks iff a version of the Kr\"{u}ll-Schmidt theorem holds for $A$, namely, that any two direct product decompositions of $A$ into irreducibles can be rearranged so that the factors are pairwise isomorphic. Likely there is a similar connection in the general setting, involving refinements of decompositions, but we have not pursued the matter. 
}
\end{rmk}

We now turn our attention to the computation of the automorphism group, and behavior of $\Gamma$ for some ``small'' sets $X$. We recall that the phase group of a structure is the kernel of the map $\Gamma$. 

\begin{prop}
Suppose $X$ is a set of prime cardinality $p$. Then $\ts{Fact}~X$ has 2 elements, its automorphism group is trivial, $\Gamma$ is onto, and the phase group of $X$ is the full symmetric group on $X$. 
\end{prop}

\begin{proof}
This is obvious from the fact that a set with a prime number of elements can only be decomposed as a direct product as a one-element set times a $p$-element set, or as a $p$-element set times a one-element set, and any two decompositions of either type are equivalent. 
\end{proof}

Recall that the orthomodular poset $\ts{MO}_n$ is a horizontal sum of $n$ copies of the 4-element Boolean algebra. So it has a bottom, a top, and an antichain of $2n$ elements in the middle paired as orthocomplements. We shall also use the following simple observation. 

\begin{lemma}
If $L$ is the horizontal sum of $k$ copies of the \ts{omp} $P$, then $\ts{Aut}~L$ is the semidirect product $(\ts{Aut}~\!P)^k\rtimes\ts{Sym}(k)$ of $k^{th}$ power of the automorphism group of $P$ by the symmetric group on $k$ letters via the obvious action. 
\label{semidirect}
\end{lemma}

We consider next our first somewhat anomalous case, that where $|X|=4$. We note that this is a prime power of $2$. The next prime power of 2, when $|X|=8$, will also provide unusual behavior. Somehow it seems there is just insufficient room in factors of 2 to behave properly. We do not know what happens when $|X|=2^4$. 

\begin{prop}
If $X$ has cardinality 4, then $\ts{Fact}~X$ is $\ts{MO}_3$, its automorphism group is $(\Bbb{Z}_2)^3\rtimes\ts{Sym}(3)$, its phase groups is the Klein four group, and $\Gamma$ is not onto. 
\end{prop}

\begin{proof}
As $|X|$ has two prime factors, all blocks have two atoms, so it is an $\ts{MO}_n$ for some $n$. It follows from either Corollary~\ref{count1} or \ref{numatoms} that $\ts{Fact}~X$ has 6 atoms, hence is $\ts{MO}_3$. Lemma~\ref{semidirect} then describes the automorphism group of $\ts{Fact}~X$ since the automorphism group of a 4-element Boolean algebra is $\Bbb{Z}_2$. That $\Gamma$ is not onto follows as $\ts{Aut}(X)$ has $4!$ elements, so is smaller than that of $\ts{Fact}~X$. It remains only to observe that the kernel of $\Gamma$ is the Klein four group. For this, one computes directly that if $X=\{a,b,c,d\}$, then the permutations $\alpha$ of $X$ with $\Gamma\alpha=id$ are $id$, and the ones whose cycle representations are $(ab)(cd)$, $(ac)(bd)$, and $(ad)(bc)$. 
\end{proof}

\begin{prop}
If $|X|=pq$ with $p,q$ prime and $p\geq 3$, then $\ts{Fact}~X$ is $\ts{MO}_n$ where 

\[ n =
\left\{
	\begin{array}{ll}
		{\displaystyle\frac{(pq)!}{p!q!}}  & \mbox{if } p \neq q \\ \\
		{\displaystyle\frac{(pq)!}{2p!q!}} & \mbox{if } p = q
	\end{array}
\right.
\]
\vspace{2ex}

\noindent The automorphism group of $\ts{Fact}~X$ is $(\Bbb{Z}_2)^n\rtimes\ts{Sym}(n)$, the phase group is trivial, and $\Gamma$ is not onto. 
\end{prop}

\begin{proof}
As $|X|$ has two prime factors, its blocks are 4-element Boolean algebras, so it is an $\ts{MO}_n$ where $n$ is the number of blocks. When $p\neq q$, 
the number of blocks equals the number of factor pairs $(\theta_1,\theta_2)$ where $\theta_1$ has $p$ blocks of $q$ elements each, as each block contains one such factor pair. By Corollary~\ref{count1} the number of blocks is as given above. When $p=q$ each block contains two factor pairs $(\theta_1,\theta_2)$ where $\theta_1$ has $p$ blocks of $p$ elements each, and is half the number of atoms. This is given by either of Corollary~\ref{count1} or \ref{numatoms} to be as above. The description of the automorphism group is given by Lemma~\ref{semidirect} as the automorphism groups of the blocks are $\Bbb{Z}_2$. Proposition~\ref{phasegroup} shows the phase group is trivial, and as the automorphism group of $\ts{Fact}~X$ has much larger cardinality than that of $X$, $\Gamma$ is not onto. 
\end{proof}


\begin{prop}
If $|X|=8$, then $\ts{Fact}~X$ is a horizontal sum of 30 copies of $\ts{Fact}~\Bbb{Z}_2^3$, its automorphism group is $(\ts{Aut}(\ts{Fact}~\!\Bbb{Z}_2^3))^{30}\rtimes\ts{Sym}(30)$. The phase group is trivial, and $\Gamma$ is not onto. 
\end{prop}

\begin{proof}
The description of $\ts{Fact}~X$ in this case is given in Proposition~\ref{8ishorizontal}, and the description of its automorphism group then follows from Lemma~\ref{semidirect}. Proposition~\ref{phasegroup} shows the phase group is trivial, and as the automorphism group of $Fact~X$ has much larger cardinality than that of $X$, $\Gamma$ is not onto. 
\end{proof}

The above result is incomplete as it describes things in terms of $\ts{Aut}(\ts{Fact}~\!\Bbb{Z}_2^3)$. However, in a nice series of papers \cite{Ovchinnikov,Chevalier1,Chevalier2,Chevalier3}, results are given that in conjunction with the fundamental theorem of projective geometry describe the automorphism group of $\ts{Fact}~V$ for any finite-dimensional vector space $V$. We briefly describe these results below. First, we note $\ts{Sub}~V$ is used to denote the subspace lattice of $V$.  

\begin{thm} \cite{Chevalier3} 
For a finite-dimensional vector space $V$, each automorphism $\sigma$ and dual automorphism $\mu$ of $\ts{Sub}~V$ give automorphisms $\sigma^*$ and $\mu^*$ of $\ts{Fact}~V$ where 

\[\sigma^*(S,T) = (\sigma S,\sigma T)\quad\mbox{ and }\quad \mu^*(S,T)=(\mu T,\mu S).\]
\vspace{.05ex}

\noindent Further, each automorphism of $\ts{Fact}~V$ arises this way. 
\label{Chevalier}
\end{thm}

It is worthwhile to briefly describe the proof of this result as it illuminates the proof in the following section. The key step is the following. Here, an element of $\ts{Fact}~V$ is said to be of height two if it covers an atom. 

\begin{lemma}
Two atoms of $\ts{Fact}~V$ have at least two distinct upper bounds of height two iff they have the same first components or the same second components. 
\end{lemma}

The proof is not difficult, relying on a bit of elementary modular lattice theory. With this result, there is a foothold on automorphisms of $\ts{Fact}~V$ as two atoms satisfying this condition must be mapped to two others satisfying it. Thus two atoms with the same first spot either get mapped to two with the same first spot, or two with the same second spot. With some work, one shows that an automorphism $\phi$ of $\ts{Fact}~V$ either takes all pairs with the same first spot to ones with the same first spot, or to ones with the same second spot. From this, it is not difficult to show that $\phi$ equal to $\alpha^*$ for some automorphism $\alpha$ of $\ts{Sub}~V$ in the first case, and equal to $\delta^*$ for some anti-automorphism of $\ts{Sub}~V$ in the second. 

\begin{thm}[The Fundamental Theorem of Projective Geometry I \cite{Baer}]
For a finite-dimensional vector space $V$, the automorphisms of $\ts{Sub}~V$ correspond to the collineations of the projective geometry associated with $V$. Each semi-linear automorphism of $V$ gives an automorphism of $\ts{Sub}~V$, all automorphisms of $\ts{Sub}~V$ arise this way, and the two semi-linear automorphisms give the same automorphism of $\ts{Sub}~V$ iff they are scalar multiples of one another. 
\end{thm}

We next introduce notation for various groups of automorphisms. 

\begin{defn}
For a finite-dimensional vector space $V$ over a field $\mathfrak{K}$, denote the groups of linear automorphisms, semilinear automorphisms, and automorphisms that are scalar multiples by $\ts{GL}(V)$, $\ts{SL}(V)$ and $\mathfrak{K}^*$ respectively. Then let $\ts{PGL}(V)$ and $\ts{PSL}(V)$ be the quotients of $\ts{GL}(V)$ and $\ts{SL}(V)$ by $\mathfrak{K}^*$. 
\end{defn}

Note $\ts{GL}(V)$ is by definition the automorphism group of $V$, and the fundamental theorem of projective geometry states $\ts{PSL}(V)$ is the automorphism group of $\ts{Sub}~V$. Clearly $\ts{Sub}~V$ is a subgroup of index 2 in the group of all automorphisms and dual automorphisms of $\ts{Sub}~V$, and by Chevalier's result, Theorem~\ref{Chevalier}, this is the automorphism group of $\ts{Fact}~V$. This gives the following. 

\begin{thm}
Let $V$ be a finite-dimensional vector space over a field $\mathfrak{K}$. The phase group of $V$ is the group $\mathfrak{K}^*$ of non-zero elements of $\mathfrak{K}$. The image of $\Gamma$ is $\ts{PGL}(V)$, and this is a subgroup of the index 2 subgroup $\ts{PSL}(V)$ of $\ts{Aut}(\ts{Fact}~\!V)$. If $\mathfrak{K}$ has no non-trivial automorphisms, then $\ts{PGL}(V)=\ts{PSL}(V)$, so the image of $\Gamma$ has index 2 in the automorphism group of $\ts{Fact}~V$. 
\end{thm}

\section{The case of a 27-element set --- the first half}

We show that for $X$ a 27-element set, the map $\Gamma$ gives an isomorphism from the permutation group of $X$ to the group of automorphisms of $\ts{Fact}~X$. The proof has two main parts. The first is to show that the automorphisms of $\ts{Fact}~X$ correspond to automorphims of the poset $\ts{Req}~X$ of regular equivalence relations of $X$. That is the content of this section. The second half is to show that automorphisms of $\ts{Req}~X$ correspond to permutations of $X$, and is in the following section. 

\begin{notation}
An equivalence relation on $X$ with 9 blocks of 3 elements each is called a small relation, and usually denoted by a lower case latin letter such as $a$. Its blocks are denoted $a_1,\ldots,a_9$. An equivalence relation with 3 blocks of 9 elements each is called a large relation and denoted by an upper case latin letter such as $A$. Its blocks are denoted $A_1,\ldots,A_3$. Ordered pairs of equivalence relations, such as $(a,A)$, will be written as $aA$. 
\end{notation}

We review some basics from Section~1 in this setting. The blocks of $\ts{Fact}~X$ all have 3 atoms, and each element of $\ts{Fact}~X$ is either a bound, an atom, or a coatom. The non-trivial factor pairs of $X$ are the $aA$ and $Aa$ where $a$ is small, $A$ is large, and $a\cap A=\Delta$ where $\Delta$ is the identity relation. Of these, the $aA$ are the atoms and the $Aa$ are the coatoms of $\ts{Fact}~X$. For atoms $aA$ and $bB$ we have $aA\perp bB$ iff $a\subseteq B$, $b\subseteq A$, and $a,b$ permute. We come to our key notion that will allow us to deal with automorphisms of $\ts{Fact}~X$. 

\begin{defn} 
Call sets of atoms $\mathfrak{X}$ and $\mathcal{Y}$ of $\ts{Fact}~X$ orthogonal, and write $\mathfrak{X}\perp\mathcal{Y}$ if each member of $\mathfrak{X}$ is orthogonal to each member of $\mathcal{Y}$. 
\end{defn}

We next turn to the results on orthogonal sets of atoms that will allow us to show that automorphisms of $\ts{Fact}~X$ act component-wise on factor pairs. In its proof, and elsewhere, we assume $X$ is the set $\{0,1,2\}^3$. We frequently draw $X$ as shown below, and use suggestive terminology such as the bottom floor, middle floor, top floor, left side wall, middle wall, right wall, front wall, back wall, in the obvious way. We refer to the elements of $X$ as strings such as $102$ rather than as ordered triples $(1,0,2)$ to aid readability of diagrams. The $x,y,z$-axes have their usual meaning. 

\begin{figure}[h]
\setlength{\unitlength}{.025in}
\begin{center}
\begin{picture}(60,65)(0,-5)
\multiput(0,0)(25,0){3}{
\multiput(0,0)(0,20){3}{\circle*{1}}
\multiput(5,5)(0,20){3}{\circle*{1}}
\multiput(10,10)(0,20){3}{\circle*{1}}
}
\put(0,0){\line(1,0){50}}
\put(10,10){\line(1,0){50}}
\put(0,40){\line(1,0){50}}
\put(10,50){\line(1,0){50}}
\put(25,0){\line(0,1){40}}
\put(35,10){\line(0,1){40}}
\put(0,0){\line(0,1){40}}
\put(10,10){\line(0,1){40}}
\put(50,0){\line(0,1){40}}
\put(60,10){\line(0,1){40}}
\put(0,0){\line(1,1){10}}
\put(0,40){\line(1,1){10}}
\put(25,0){\line(1,1){10}}
\put(25,40){\line(1,1){10}}
\put(50,0){\line(1,1){10}}
\put(50,40){\line(1,1){10}}
\put(-5,-3){\makebox(0,0)[rt]{$200$}}
\put(65,10){\makebox(0,0)[l]{$020$}}
\put(12,15){\makebox(0,0)[bl]{$000$}}
\put(10,55){\makebox(0,0)[b]{$002$}}
\end{picture}
\end{center}
\label{fig:X}
\caption{$X$}
\end{figure}

\begin{prop}
Suppose $a,b$ are small, permute, and have $a\cap b = \Delta$. Set 
\begin{eqnarray*}
\mathfrak{X}(a,b)&=&\{aA:aA\mbox{ is an atom and }b\subseteq A\}\\
\mathcal{Y}(a,b)&=&\{bB:bB\mbox{ is an atom and }a\subseteq B\}
\end{eqnarray*}
Then $\mathfrak{X}(a,b)$ and $\mathcal{Y}(a,b)$ each have 36 elements and $\mathfrak{X}(a,b)\perp\mathcal{Y}(a,b)$. Further, the intersection of the second spots, $A$, of members of $\mathfrak{X}(a,b)$ is $b$, and the intersection of the second spots, $B$, of members of $\mathcal{Y}(a,b)$ is $a$. 
\label{llama}
\end{prop}

\begin{proof}
Assume blocks of $a$ are lines parallel to the $z$-axis, and blocks of $b$ are lines parallel to the $y$-axis. As $a,b$ are small, permute, and $a\cap b=\Delta$ we have $a\circ b$ is an equivalence relation that must be large, in this case its blocks are the planes $x=0$, $x=1$ and $x=2$, i.e. the front wall, middle wall, and back wall. It follows that this assumption on the choice of $a,b$ general. 

The large $A$ with $aA\in\mathfrak{X}(a,b)$ are those with $b\subseteq A$ and $a\cap A=\Delta$. These conditions mean that each block of $A$ consists of 3 lines in the $y$-direction ($b$-blocks) and that no two of these lines are vertical translates of one another (since $A\cap a=\Delta$). So each of the blocks of $A$ contains exactly one of $000$, $001$, and $002$. 

The block of $A$ containing $000$ must contain the line in $y$-direction with $000$, as well as one line $y$-direction from the plane $x=1$, and one line $y$-direction from the plane $x=2$. So there are $3\cdot 3$ ways to construct this block of $A$. Assuming the block containing $000$ is chosen, we consider the block of $A$ containing $001$. This block contains the line $y$-direction with $001$ as well as one of the two remaining lines $y$-direction from the plane $x=1$, and one of the two remaining lines $y$-direction from the plane $x=2$. So there are $2\cdot 2$ ways to construct this block. Then the final block of $A$ is determined. In all, there are $3\cdot 3\cdot 2\cdot 2 = 36$ ways to construct $A$. 

We have shown that $\mathfrak{X}(a,b)$ has 36 elements, and by symmetry so does $\mathcal{Y}(a,b)$. To see $\mathfrak{X}(a,b)\perp\mathcal{Y}(a,b)$ suppose $aA\in\mathfrak{X}(a,b)$ and $bB\in\mathcal{Y}(a,b)$. By definition $a\subseteq B$, $b\subseteq A$, and we began by assuming $a,b$ permute. Thus $aA\perp bB$. The description of the large $A$ with $aA\in\mathfrak{X}(a,b)$ shows their intersection is $b$, and this, with its dual statement, gives the further condition of the proposition. 
\end{proof}

We next show that these simple conditions characterize the sets arising as $\mathfrak{X}(a,b)$ and $\mathcal{Y}(a,b)$. This proof will take some effort and is spread through a number of claims across several pages. 

\begin{prop}
If $\mathfrak{X},\mathcal{Y}$ are two sets of 36 atoms each with $\mathfrak{X}\perp\mathcal{Y}$, then there are small permuting $a,b$ with $a\cap b=\Delta$ such that $\mathfrak{X}=\mathfrak{X}(a,b)$ and $\mathcal{Y}=\mathcal{Y}(a,b)$. 
\label{main}
\end{prop}

\begin{proof}
Suppose the atoms in $\mathfrak{X}$ are $x_iX_i$ and those in $\mathcal{Y}$ are $y_jY_j$ for $i,j=1,\ldots,36$. We assume further that $x_1X_1$ is the factor pair where the $x_1$ blocks are lines parallel to the $z$-axis and the $X_1$ blocks are the planes $z=0$, $z=1$, $z=2$; and $y_1Y_1$ is the factor pair where the $y_1$ blocks are the lines parallel to the $y$-axis and $Y_1$ blocks are the planes $y=0$, $y=1$, $y=2$. 

\begin{figure}[h]
\setlength{\unitlength}{.022in}
\begin{center}
\begin{picture}(60,65)(0,-10)
\multiput(0,0)(25,0){3}{
\multiput(0,0)(0,20){3}{\circle*{1.5}}
\multiput(5,5)(0,20){3}{\circle*{1.5}}
\multiput(10,10)(0,20){3}{\circle*{1.5}}
}
\multiput(0,0)(25,0){3}{
\multiput(0,0)(5,5){3}{\line(0,1){40}}
}
\end{picture}
\hspace{10ex}
\begin{picture}(60,65)(0,-10)
\multiput(0,0)(25,0){3}{
\multiput(0,0)(0,20){3}{\circle*{1.5}}
\multiput(5,5)(0,20){3}{\circle*{1.5}}
\multiput(10,10)(0,20){3}{\circle*{1.5}}
}
\multiput(0,0)(0,20){3}{
\put(0,0){\line(1,0){50}}
\put(10,10){\line(1,0){50}}
\put(0,0){\line(1,1){10}}
\put(50,0){\line(1,1){10}}
}
\end{picture}
\end{center}
\caption{$x_1$ and $X_1$}
\end{figure}
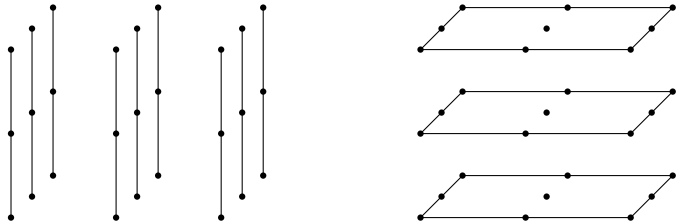

\begin{figure}[h]
\setlength{\unitlength}{.022in}
\begin{center}
\begin{picture}(60,65)(0,-10)
\multiput(0,0)(25,0){3}{
\multiput(0,0)(0,20){3}{\circle*{1.5}}
\multiput(5,5)(0,20){3}{\circle*{1.5}}
\multiput(10,10)(0,20){3}{\circle*{1.5}}
}
\multiput(0,0)(0,20){3}{
\multiput(0,0)(5,5){3}{\line(1,0){50}}
}
\end{picture}
\hspace{10ex}
\begin{picture}(60,65)(0,-10)
\multiput(0,0)(25,0){3}{
\multiput(0,0)(0,20){3}{\circle*{1.5}}
\multiput(5,5)(0,20){3}{\circle*{1.5}}
\multiput(10,10)(0,20){3}{\circle*{1.5}}
}
\multiput(0,0)(25,0){3}{
\put(0,0){\line(0,1){40}}
\put(0,0){\line(1,1){10}}
\put(10,10){\line(0,1){40}}
\put(0,40){\line(1,1){10}}
}
\end{picture}
\end{center}
\caption{$y_1$ and $Y_1$}
\end{figure}
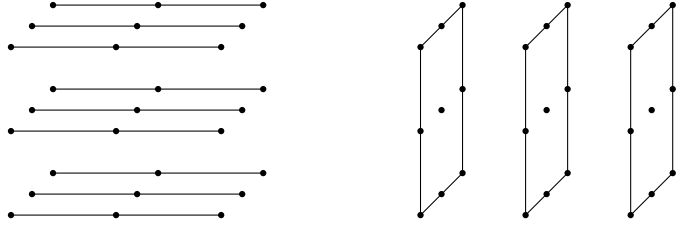

\begin{claim}
There are 36 small $u$ with $uY_1$ an atom orthogonal to $x_1X_1$, and the transitive closure of the union of these $u$'s is $X_1$. 
\label{C1}
\end{claim}

\begin{proof}[Proof of Claim:] 
Such $u$ must be small, permute with $x_1$, be contained in $X_1$ and intersect $Y_1$ trivially. As $u$ is contained in $X_1$ it partitions each block of $X_1$ (the planes $z=0,1,2$) into three pieces of three, and as $u$ permutes with $x_1$ (whose blocks are vertical lines), the partitions of these three planes must be vertical translates of one another. So $u$ is completely determined by its intersection with the plane $z=0$. As $u$ intersects $Y_1$ trivially, none of the blocks of $u$ in the plane $z=0$ contains two elements with the same $y$-coordinate. So each of these three blocks contains exactly one of $000$, $100$, $200$. There are $3\cdot 3$ ways to choose the block containing $000$, then $2\cdot 2$ ways to choose the block containing $100$. This then determines $u$, so there are 36 such $u$. It is clear from the description of the construction of such $u$'s that the transitive closure of their union is $X_1$. 
\end{proof}

\begin{claim}
It cannot happen that all $x_iX_i$ or all $y_jY_j$ have the same second spot. 
\label{gn} 
\end{claim}

\begin{proof}[Proof of Claim:] 
Suppose all the $y_jY_j$ have the same second spot, which will be $Y_1$. As each $y_jY_j$ is an atom orthogonal to $x_1X_1$, the $y_j$ for $j=1,\ldots,36$ must be the 36 $u$'s of Claim~\ref{C1}. Since $x_iX_i\perp y_jY_j$ for each $i,j$, we have $y_j\subseteq X_i$ for each $i,j$, and as the $y_j$ are the 36 $u$'s, and the transitive closure of these 36 $u$'s is $X_1$, it follows that $X_i=X_1$ for each $i$. Using a symmetric version of Claim~\ref{C1}, there are 36 small $v$'s with $vX_1$ an atom orthogonal to $y_1Y_1$, and these $v$'s must be the $x_i$. These $v$'s are the small congruences that partition blocks of $Y_1$ (the planes $y=0,1,2$) into blocks of three, are disjoint from $X_1$, and whose partitions of the different blocks of $Y_1$ are translates of one another along the $y$-axis. One such choice of $u,v$ is shown below. 

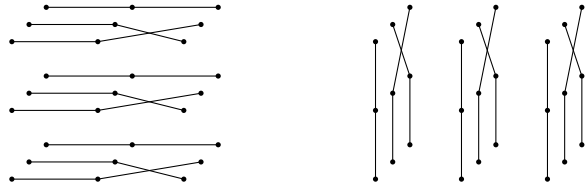
\begin{figure}[h!tb]
\setlength{\unitlength}{.018in}
\begin{center}
\begin{picture}(60,55)(0,0)
\multiput(0,0)(25,0){3}{
\multiput(0,0)(0,20){3}{\circle*{1.5}}
\multiput(5,5)(0,20){3}{\circle*{1.5}}
\multiput(10,10)(0,20){3}{\circle*{1.5}}
}
\multiput(0,0)(0,20){3}{
\put(0,0){\line(1,0){25}}
\put(5,5){\line(1,0){25}}
\put(10,10){\line(1,0){50}}
\put(25,0){\line(6,1){30}}
\put(30,5){\line(4,-1){20}}
}
\end{picture}
\hspace{10ex}
\begin{picture}(60,55)(0,0)
\multiput(0,0)(25,0){3}{
\multiput(0,0)(0,20){3}{\circle*{1.5}}
\multiput(5,5)(0,20){3}{\circle*{1.5}}
\multiput(10,10)(0,20){3}{\circle*{1.5}}
}
\multiput(0,0)(25,0){3}{
\put(5,5){\line(0,1){20}}
\put(10,10){\line(0,1){20}}
\put(0,0){\line(0,1){40}}
\put(5,25){\line(1,5){5}}
\put(10,30){\line(-1,3){5}}
}
\end{picture}
\end{center}
\caption{A choice of $u$ and $v$}
\end{figure}

Note that the transitive closure of the union of the blocks of $u,v$ containing 200 has more than 9 elements, and as $u\cap v=\Delta$, this implies they do not permute. This provides a contradiction to the fact that $x_iX_i\perp y_jY_j$, since this implies $x_i$ permutes with $y_j$ for each $i,j$. 
\end{proof}

If all the $x_iX_i$ have the same first component, they must all have different second components, and these second components must all contain $y_1$. This implies that $\mathfrak{X}$ is contained in $\mathfrak{X}(x_1,y_1)$ and as both have 36 elements, they are equal. Proposition~\ref{llama} gives that the intersection of the second members of $\mathfrak{X}(x_1,y_1)$ equals $y_1$. As this intersection contains each $y_j$, this implies that all $y_jY_j$ have the same first component, hence $\mathcal{Y}=\mathcal{Y}(x_1,y_1)$. This would establish our result, as would the dual argument if  all $y_jY_j$ have the same first components. Modulo a renumbering of the elements, we then have the following result that we will use to argue by contradiction.

\begin{claim} 
If Proposition~\ref{main} is not true, then  $x_1\neq x_2$ and $y_1\neq y_2$. 
\label{io}
\end{claim}

Since we have  $x_iX_i\perp y_1Y_1$, the blocks of $x_i$ are contained in the blocks of $Y_1$ (the planes $y=0,1,2$) and the blocks of $x_i$ contained in one of these planes are translates in the $y$-direction (since blocks of $y_1$ are lines in the $y$-direction) of the blocks contained in another plane. Similarly, the blocks of $y_j$ are contained in the blocks of $X_1$ (the planes $z=0,1,2$) and the blocks of $y_j$ contained in one of these planes are translates in the $z$-direction (since blocks of $x_1$ are lines in the $z$-direction) of the blocks contained in another plane. 

\begin{claim}
If $x_i\neq x_1$ and $y_j\neq y_1$, then the transitive closures $\textsc{Trcl}(x_1\cup x_i)$ and $\textsc{Trcl}(y_1\cup y_2)$ each have three blocks of three and three blocks of six.
\label{mk}
\end{claim}

\begin{proof}[Proof of Claim:]
We know $x_1\neq x_i$, both are contained in $Y_1$, and both are determined by their intersection with one of the blocks (planes) of $Y_1$. So $\textsc{Trcl}(x_1\cup x_i)$ is also determined by its intersection with one block of $Y_1$. As the transitive closure coalesces blocks of $x_1$, this intersection with a block of $Y_1$ can have (i) three blocks of three, (ii) one block of three and one block of six, or (iii) one block of nine. The first case gives $x_1=x_i$. The third implies the transitive closure is $Y_1$. As $x_1,x_i\subseteq Y_j$ for each $j$, we would then have $Y_1=Y_j$ for all $j$, contrary to Claim~\ref{gn}. So if $x_1\neq x_i$, the second case must hold. The argument for $y_1\neq y_j$ is similar. 
\end{proof}

By Claim~\ref{gn} each $y_j$ must contain a block of $y_1$ that is contained in the plane $z=0$, hence must contain one of the lines in the $y$-direction containing $000$, $100$ or $200$. There is symmetry to the situation, and we make the following assumption. 

\begin{assumption}
The relation $y_2$ contains $000, 010, 020$. 
\label{ass}
\end{assumption}

Consider the small relations $\beta$ that are contained in $X_1$, permute with $x_1$, and contain the block $000, 010, 020$. As such $\beta$ are determined by their intersection with the block (plane) $z=0$ of $X_1$,  they partition the remaining six elements of this plane into two blocks of three. There are ten ways to do this, shown below as $\beta_1,\ldots,\beta_{10}$ with $\beta_1=y_1$.  The discussion above provides the following claim. 

\begin{figure}[h]
\setlength{\unitlength}{.012in}
\begin{center}
\begin{picture}(60,65)(0,-10)
\multiput(0,0)(25,0){3}{
\multiput(0,0)(0,20){3}{\circle*{1.5}}
\multiput(5,5)(0,20){3}{\circle*{1.5}}
\multiput(10,10)(0,20){3}{\circle*{1.5}}
}
\multiput(0,0)(0,20){3}{
\put(0,0){\line(1,0){50}}
\put(5,5){\line(1,0){50}}
\put(10,10){\line(1,0){50}}}
\put(30,-10){\makebox(0,0)[t]{$\beta_1$}}
\end{picture}
\hspace{3ex}
\begin{picture}(60,65)(0,-10)
\multiput(0,0)(25,0){3}{
\multiput(0,0)(0,20){3}{\circle*{1.5}}
\multiput(5,5)(0,20){3}{\circle*{1.5}}
\multiput(10,10)(0,20){3}{\circle*{1.5}}
}
\multiput(0,0)(0,20){3}{
\put(0,0){\line(1,0){25}}
\put(5,5){\line(1,0){25}}
\put(10,10){\line(1,0){50}}
\put(25,0){\line(6,1){30}}
\put(30,5){\line(4,-1){20}}
}
\put(30,-10){\makebox(0,0)[t]{$\beta_2$}}
\end{picture}
\hspace{3ex}
\begin{picture}(60,65)(0,-10)
\multiput(0,0)(25,0){3}{
\multiput(0,0)(0,20){3}{\circle*{1.5}}
\multiput(5,5)(0,20){3}{\circle*{1.5}}
\multiput(10,10)(0,20){3}{\circle*{1.5}}
}
\multiput(0,0)(0,20){3}{
\put(0,0){\line(6,1){30}}
\put(5,5){\line(4,-1){20}}
\put(10,10){\line(1,0){50}}
\put(25,0){\line(6,1){30}}
\put(30,5){\line(4,-1){20}}
}
\put(30,-10){\makebox(0,0)[t]{$\beta_3$}}
\end{picture}
\hspace{3ex}
\begin{picture}(60,65)(0,-10)
\multiput(0,0)(25,0){3}{
\multiput(0,0)(0,20){3}{\circle*{1.5}}
\multiput(5,5)(0,20){3}{\circle*{1.5}}
\multiput(10,10)(0,20){3}{\circle*{1.5}}
}
\multiput(0,0)(0,20){3}{
\put(0,0){\line(6,1){30}}
\put(5,5){\line(4,-1){20}}
\put(10,10){\line(1,0){50}}
\put(25,0){\line(1,0){25}}
\put(30,5){\line(1,0){25}}
}
\put(30,-10){\makebox(0,0)[t]{$\beta_4$}}
\end{picture}

\end{center}
\end{figure}

\begin{figure}[h]
\setlength{\unitlength}{.012in}
\begin{center}
\begin{picture}(60,65)(0,-10)
\multiput(0,0)(25,0){3}{
\multiput(0,0)(0,20){3}{\circle*{1.5}}
\multiput(5,5)(0,20){3}{\circle*{1.5}}
\multiput(10,10)(0,20){3}{\circle*{1.5}}
}
\multiput(0,0)(0,20){3}{
\put(25,0){\line(1,0){25}}
\put(10,10){\line(1,0){50}}
\put(0,0){\line(1,1){5}}
\put(25,0){\line(1,1){5}}
\put(0,0){\qbezier(5,5)(30,9.5)(55,5)}
}
\put(30,-10){\makebox(0,0)[t]{$\beta_5$}}
\end{picture}
\hspace{3ex}
\begin{picture}(60,65)(0,-10)
\multiput(0,0)(25,0){3}{
\multiput(0,0)(0,20){3}{\circle*{1.5}}
\multiput(5,5)(0,20){3}{\circle*{1.5}}
\multiput(10,10)(0,20){3}{\circle*{1.5}}
}
\multiput(0,0)(0,20){3}{
\put(30,5){\line(1,0){25}}
\put(10,10){\line(1,0){50}}
\put(0,0){\line(1,1){5}}
\put(25,0){\line(1,1){5}}
\put(0,0){\qbezier(0,0)(20,-5.5)(50,0)}
}
\put(30,-10){\makebox(0,0)[t]{$\beta_6$}}
\end{picture}
\hspace{3ex}
\begin{picture}(60,65)(0,-10)
\multiput(0,0)(25,0){3}{
\multiput(0,0)(0,20){3}{\circle*{1.5}}
\multiput(5,5)(0,20){3}{\circle*{1.5}}
\multiput(10,10)(0,20){3}{\circle*{1.5}}
}
\multiput(0,0)(0,20){3}{
\put(30,5){\line(1,0){25}}
\put(10,10){\line(1,0){50}}
\put(0,0){\line(1,1){5}}
\put(50,0){\line(1,1){5}}
\put(0,0){\line(1,0){25}}
}
\put(30,-10){\makebox(0,0)[t]{$\beta_7$}}
\end{picture}
\hspace{3ex}
\begin{picture}(60,65)(0,-10)
\multiput(0,0)(25,0){3}{
\multiput(0,0)(0,20){3}{\circle*{1.5}}
\multiput(5,5)(0,20){3}{\circle*{1.5}}
\multiput(10,10)(0,20){3}{\circle*{1.5}}
}
\multiput(0,0)(0,20){3}{
\put(25,0){\line(1,0){25}}
\put(10,10){\line(1,0){50}}
\put(0,0){\line(1,1){5}}
\put(50,0){\line(1,1){5}}
\put(5,5){\line(1,0){25}}
}
\put(30,-10){\makebox(0,0)[t]{$\beta_8$}}
\end{picture}
\hspace{3ex}
\begin{picture}(60,65)(0,-10)
\multiput(0,0)(25,0){3}{
\multiput(0,0)(0,20){3}{\circle*{1.5}}
\multiput(5,5)(0,20){3}{\circle*{1.5}}
\multiput(10,10)(0,20){3}{\circle*{1.5}}
}
\multiput(0,0)(0,20){3}{
\put(0,0){\qbezier(5,5)(30,9.5)(55,5)}
\put(10,10){\line(1,0){50}}
\put(25,0){\line(1,1){5}}
\put(50,0){\line(1,1){5}}
\put(0,0){\line(1,0){25}}
}
\put(30,-10){\makebox(0,0)[t]{$\beta_9$}}
\end{picture}
\hspace{3ex}
\begin{picture}(60,65)(0,-10)
\multiput(0,0)(25,0){3}{
\multiput(0,0)(0,20){3}{\circle*{1.5}}
\multiput(5,5)(0,20){3}{\circle*{1.5}}
\multiput(10,10)(0,20){3}{\circle*{1.5}}
}
\multiput(0,0)(0,20){3}{
\put(5,5){\line(1,0){25}}
\put(10,10){\line(1,0){50}}
\put(50,0){\line(1,1){5}}
\put(25,0){\line(1,1){5}}
\put(0,0){\qbezier(0,0)(20,-5.5)(50,0)}
}
\put(30,-10){\makebox(0,0)[t]{$\beta_{10}$}}
\end{picture}

\end{center}
\end{figure}

\begin{claim}
The relation $y_2$, and any other $y_j$ that contains the block $000, 010, 020$, is one of $\beta_1,\ldots,\beta_{10}$. 
\end{claim}


For any $i$ we have $x_i\subseteq Y_1$, that $x_i$ permutes with $y_1,y_2$, and if $x_1\neq x_i$ then $\textsc{Trcl}(x_1\cup x_i)$ has three blocks of three and three blocks of six. So the blocks of $x_i$ are contained in the blocks (planes) of $Y_1$, and as $x_i$ permutes with $y_1$, the blocks of $x_i$ in one plane of $Y_1$ are translates in the $y$-direction of the blocks of $x_i$ in another plane of $Y_1$. So $x_i$ is determined by its intersection with the plane $y=0$, and the description of $\textsc{Trcl}(x_1\cup x_i)$ shows $x_i$ contains one of the blocks of $x_1$ in this plane. 

\begin{claim}
Under Assumption~\ref{ass}, all the $x_i$ contain $000, 001, 002$.
\label{lko}
\end{claim}

\begin{proof}[Proof of Claim:] 
We know each $x_i$ contains a block of $x_1$ in the plane $y=0$. Suppose this is the block is the vertical line above $200$. As $x_i$ permutes with $y_1$, it also contains the vertical lines above $210$ and above $220$. If $x_i\neq x_1$, then as $x_i$ is determined by its intersection with the plane $y=0$, there are $u,v$ with $x_i$ relating $00u$ and $10v$. As the blocks of $x_i$ in the planes $y=1, y=2$ are translates in the $y$-direction of its blocks in the plane $y=0$, $x_i$ relates $0mu$ and $1mv$ for $m=0,1,2$. 

We know $y_2$ is one of $\beta_2,\ldots,\beta_{10}$. Inspecting these $\beta_j$ there are some $a,b$ with $y_2$ relating $1a0$ and $2b0$. Let $cd0$ be the third point in this block of $y_2$. Consider the equivalence class $S$ of $\textsc{Trcl}(x_i\cup y_2)$ that contains $2b0$. As the vertical line above $2b0$ is a block of $x_i$, and $\{1an, 2bn, cdn\}$ is a block of $y_2$ for each $n=0,1,2$, the class $S$ contains the nine elements $\{1an, 2bn, cdn:n=0,1,2\}$. 

Setting $m=a$ shows $x_i$ relates $0au$ and $1av$. But $1av$ belongs to $\textsc{Trcl}(x_i\cup y_2)$, hence so does $0au$. This shows $S$ has more than 9 elements, and this implies $x_i$ does not permute with $y_2$, a contradiction. Thus if $x_i$ contains the vertical line above $200$ it must be $x_1$, and a similar argument shows that if $x_i$ contains the vertical line above $100$ it is $x_1$. Therefore, if $x_i\neq x_1$ it must contain the vertical line above $000$. 
\end{proof}

\begin{claim}
Under Assumption~\ref{ass}, all the $y_j$ contain the block $000, 010, 020$, hence all are among $\beta_1,\ldots,\beta_{10}$. 
\end{claim}

\begin{proof}[Proof of Claim:] 
By Claim~\ref{mk}, each $y_j$ contains a line in the $y$-direction containing either $000, 100$ or $200$. The arguments in Claim~\ref{lko} show that if $y_2$ contains the line with $000$, then all the $x_i$ contain the vertical line with $000$. Symmetry clearly shows that if $y_2$ contains the horizontal line with $n00$, then all $x_i$ contain the vertical line with $n00$. In any case, all $x_i\neq x_1$ contain the same vertical line. Clearly the dual argument shows all $y_j\neq y_1$ contain the same horizontal line. 
\end{proof}

There are ten small $\alpha$ that are contained in $Y_1$, permute with $y_1$, and contain the block $000, 001, 002$. These shown below as $\alpha_1,\ldots,\alpha_{10}$ with $\alpha_1=x_1$. The above results show all the $x_i$ are among $\alpha_1,\ldots,\alpha_{10}$, and all the $y_j$ are among $\beta_1,\ldots,\beta_{10}$. 

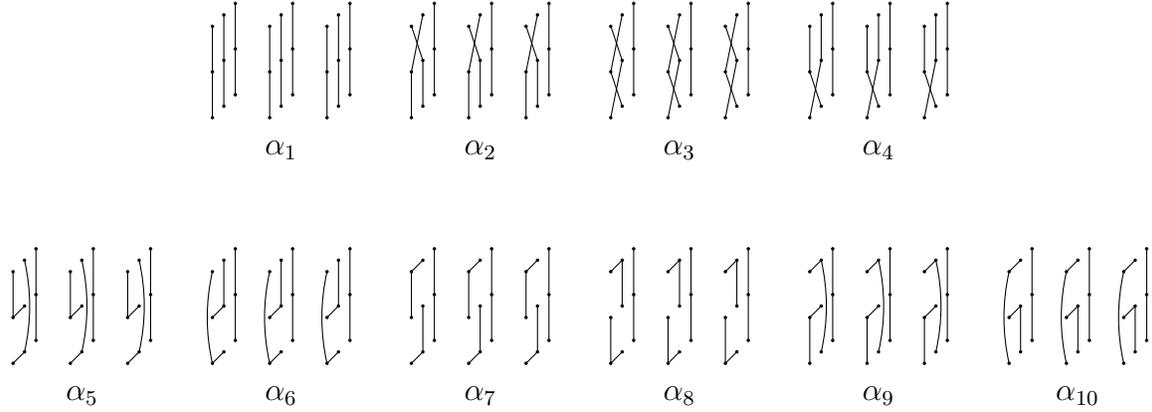
\begin{figure}[h]
\setlength{\unitlength}{.012in}
\begin{center}
\begin{picture}(60,65)(0,-10)
\multiput(0,0)(25,0){3}{
\multiput(0,0)(0,20){3}{\circle*{1.5}}
\multiput(5,5)(0,20){3}{\circle*{1.5}}
\multiput(10,10)(0,20){3}{\circle*{1.5}}
}
\multiput(0,0)(25,0){3}{
\put(0,0){\line(0,1){40}}
\put(5,5){\line(0,1){40}}
\put(10,10){\line(0,1){40}}
}
\put(30,-10){\makebox(0,0)[t]{$\alpha_1$}}
\end{picture}
\hspace{3ex}
\begin{picture}(60,65)(0,-10)
\multiput(0,0)(25,0){3}{
\multiput(0,0)(0,20){3}{\circle*{1.5}}
\multiput(5,5)(0,20){3}{\circle*{1.5}}
\multiput(10,10)(0,20){3}{\circle*{1.5}}
}
\multiput(0,0)(25,0){3}{
\put(0,20){\line(1,5){5}}
\put(5,25){\line(-1,3){5}}
\put(0,0){\line(0,1){20}}
\put(5,5){\line(0,1){20}}
\put(10,10){\line(0,1){40}}
}
\put(30,-10){\makebox(0,0)[t]{$\alpha_2$}}
\end{picture}
\hspace{3ex}
\begin{picture}(60,65)(0,-10)
\multiput(0,0)(25,0){3}{
\multiput(0,0)(0,20){3}{\circle*{1.5}}
\multiput(5,5)(0,20){3}{\circle*{1.5}}
\multiput(10,10)(0,20){3}{\circle*{1.5}}
}
\multiput(0,0)(25,0){3}{
\put(0,20){\line(1,5){5}}
\put(5,25){\line(-1,3){5}}
\put(0,0){\line(1,5){5}}
\put(5,5){\line(-1,3){5}}
\put(10,10){\line(0,1){40}}
}
\put(30,-10){\makebox(0,0)[t]{$\alpha_3$}}
\end{picture}
\hspace{3ex}
\begin{picture}(60,65)(0,-10)
\multiput(0,0)(25,0){3}{
\multiput(0,0)(0,20){3}{\circle*{1.5}}
\multiput(5,5)(0,20){3}{\circle*{1.5}}
\multiput(10,10)(0,20){3}{\circle*{1.5}}
}
\multiput(0,0)(25,0){3}{
\put(0,20){\line(0,1){20}}
\put(5,25){\line(0,1){20}}
\put(0,0){\line(1,5){5}}
\put(5,5){\line(-1,3){5}}
\put(10,10){\line(0,1){40}}
}
\put(30,-10){\makebox(0,0)[t]{$\alpha_4$}}
\end{picture}

\end{center}
\end{figure}

\begin{figure}[h]
\setlength{\unitlength}{.012in}
\begin{center}
\begin{picture}(60,65)(0,-10)
\multiput(0,0)(25,0){3}{
\multiput(0,0)(0,20){3}{\circle*{1.5}}
\multiput(5,5)(0,20){3}{\circle*{1.5}}
\multiput(10,10)(0,20){3}{\circle*{1.5}}
}
\multiput(0,0)(25,0){3}{
\put(0,0){\line(1,1){5}}
\put(0,20){\line(1,1){5}}
\put(0,20){\line(0,1){20}}
\put(0,0){\qbezier(5,5)(9.5,25)(5,45)}
\put(10,10){\line(0,1){40}}
}
\put(30,-10){\makebox(0,0)[t]{$\alpha_5$}}
\end{picture}
\hspace{3ex}
\begin{picture}(60,65)(0,-10)
\multiput(0,0)(25,0){3}{
\multiput(0,0)(0,20){3}{\circle*{1.5}}
\multiput(5,5)(0,20){3}{\circle*{1.5}}
\multiput(10,10)(0,20){3}{\circle*{1.5}}
}
\multiput(0,0)(25,0){3}{
\put(0,0){\line(1,1){5}}
\put(0,20){\line(1,1){5}}
\put(5,25){\line(0,1){20}}
\put(0,0){\qbezier(0,0)(-4.5,20)(0,40)}
\put(10,10){\line(0,1){40}}
}
\put(30,-10){\makebox(0,0)[t]{$\alpha_6$}}
\end{picture}
\hspace{3ex}
\begin{picture}(60,65)(0,-10)
\multiput(0,0)(25,0){3}{
\multiput(0,0)(0,20){3}{\circle*{1.5}}
\multiput(5,5)(0,20){3}{\circle*{1.5}}
\multiput(10,10)(0,20){3}{\circle*{1.5}}
}
\multiput(0,0)(25,0){3}{
\put(0,0){\line(1,1){5}}
\put(0,40){\line(1,1){5}}
\put(5,5){\line(0,1){20}}
\put(0,20){\line(0,1){20}}
\put(10,10){\line(0,1){40}}
}
\put(30,-10){\makebox(0,0)[t]{$\alpha_7$}}
\end{picture}
\hspace{3ex}
\begin{picture}(60,65)(0,-10)
\multiput(0,0)(25,0){3}{
\multiput(0,0)(0,20){3}{\circle*{1.5}}
\multiput(5,5)(0,20){3}{\circle*{1.5}}
\multiput(10,10)(0,20){3}{\circle*{1.5}}
}
\multiput(0,0)(25,0){3}{
\put(0,0){\line(1,1){5}}
\put(0,40){\line(1,1){5}}
\put(0,0){\line(0,1){20}}
\put(5,25){\line(0,1){20}}
\put(10,10){\line(0,1){40}}
}
\put(30,-10){\makebox(0,0)[t]{$\alpha_8$}}
\end{picture}
\hspace{3ex}
\begin{picture}(60,65)(0,-10)
\multiput(0,0)(25,0){3}{
\multiput(0,0)(0,20){3}{\circle*{1.5}}
\multiput(5,5)(0,20){3}{\circle*{1.5}}
\multiput(10,10)(0,20){3}{\circle*{1.5}}
}
\multiput(0,0)(25,0){3}{
\put(0,40){\line(1,1){5}}
\put(0,20){\line(1,1){5}}
\put(0,00){\line(0,1){20}}
\put(0,0){\qbezier(5,5)(9.5,25)(5,45)}
\put(10,10){\line(0,1){40}}
}
\put(30,-10){\makebox(0,0)[t]{$\alpha_9$}}
\end{picture}
\hspace{3ex}
\begin{picture}(60,65)(0,-10)
\multiput(0,0)(25,0){3}{
\multiput(0,0)(0,20){3}{\circle*{1.5}}
\multiput(5,5)(0,20){3}{\circle*{1.5}}
\multiput(10,10)(0,20){3}{\circle*{1.5}}
}
\multiput(0,0)(25,0){3}{
\put(0,40){\line(1,1){5}}
\put(0,20){\line(1,1){5}}
\put(5,5){\line(0,1){20}}
\put(0,0){\qbezier(0,0)(-4.5,20)(0,40)}
\put(10,10){\line(0,1){40}}
}
\put(30,-10){\makebox(0,0)[t]{$\alpha_{10}$}}
\end{picture}

\end{center}
\end{figure}

\begin{claim}
The $x_i$ are among $\alpha_1,\ldots,\alpha_4$ and the $y_j$ are among $\beta_1,\ldots,\beta_4$. 
\label{wah}
\end{claim}

\begin{proof}[Proof of Claim:]
If we consider $\textsc{Trcl}(\alpha_i\cup \beta_j)$ for $i=5,\ldots,10$ and $j=2,\ldots,4$ we see that the block containing $200$ consists of all elements in the planes $x=1, x=2$. It follows that $\alpha_i$ does not permute with $\beta_j$. If we consider $\textsc{Trcl}(\alpha_i\cup\beta_j)$ for $i=2,\ldots,4$ and $j=5,\ldots,10$ we see that the block containing $200$ again consists of all elements in the planes $x=1,x=2$. So in this case also $\alpha_i$ does not permute with $\beta_j$. Note also that $\alpha_i\cap\beta_j\neq\Delta$ when $i=5,\ldots,10$ and $j=5,\ldots,10$. So if some $x_i$ is one of $\alpha_5,\ldots,\alpha_{10}$, then no $y_j$ can be one of $\beta_2,\ldots,\beta_{10}$ since for each $x_i$ and $y_j$ we have $x_i$ permutes with $y_j$ and $x_i\cap y_j=\Delta$. But this is contrary to some $y_j\neq y_1$, so no $x_i$ can be among $\alpha_5,\ldots,\alpha_{10}$. Similarly no $y_j$ can be among $\beta_5,\ldots,\beta_{10}$. 
\end{proof}

\begin{claim}
Each of $\mathfrak{X}$ and $\mathcal{Y}$ have at most 24 elements. 
\end{claim}

\begin{proof}[Proof of Claim:]
By Claim~\ref{wah} we have $x_2=\alpha_i$ and $y_2=\beta_j$ for some $i,j=2,3,4$. The the diagram below shows $\textsc{Trcl}(x_1\cup x_2)$ and $\textsc{Trcl}(y_1\cup y_2)$. 

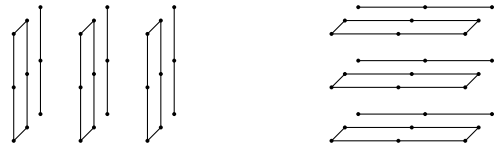
\begin{figure}[h]
\setlength{\unitlength}{.014in}
\begin{center}
\begin{picture}(60,50)(0,0)
\multiput(0,0)(25,0){3}{
\multiput(0,0)(0,20){3}{\circle*{1.5}}
\multiput(5,5)(0,20){3}{\circle*{1.5}}
\multiput(10,10)(0,20){3}{\circle*{1.5}}
}
\multiput(0,0)(25,0){3}{
\put(0,0){\line(1,1){5}}
\put(5,5){\line(0,1){40}}
\put(0,0){\line(0,1){40}}
\put(0,40){\line(1,1){5}}
\put(10,10){\line(0,1){40}}
}
\end{picture}
\hspace{10ex}
\begin{picture}(60,50)(0,0)
\multiput(0,0)(25,0){3}{
\multiput(0,0)(0,20){3}{\circle*{1.5}}
\multiput(5,5)(0,20){3}{\circle*{1.5}}
\multiput(10,10)(0,20){3}{\circle*{1.5}}
}
\multiput(0,0)(0,20){3}{
\put(0,0){\line(1,0){50}}
\put(5,5){\line(1,0){50}}
\put(0,0){\line(1,1){5}}
\put(50,0){\line(1,1){5}}
\put(10,10){\line(1,0){50}}
}
\end{picture}
\end{center}
\end{figure}

\noindent Each $Y_j$ contains $\textsc{Trcl}(x_1\cup x_2)$ and each $X_i$ contains $\textsc{Trcl}(y_1\cup y_2)$. To find a large relation that contains $\textsc{Trcl}(x_1\cup x_2)$ we must pair the three vertical lines with the three blocks of six. So there are six ways to construct a large relation containing $\textsc{Trcl}(x_1\cup x_2)$, and similarly six large relations containing $\textsc{Trcl}(y_1\cup y_2)$. So there are at most six different $X_i$ and at most six different $Y_j$. As there are at most four choices for $x_i$ and four for the $y_j$ our claim is proved. 
\end{proof}

This contradiction proves Proposition~\ref{main}. 
\end{proof}

We now begin the task of using Proposition~\ref{main} to work with automorphisms. For this, a different method of representing $X$ is convenient for a number of the proofs. We assume the elements of $X$ are $1,\ldots,27$ and indicate a factor pair $aA$ of $X$ by arranging the numbers $1,\ldots,27$ in three columns of 9 elements each, with the rows of the array being the blocks of the small relation $a$, and the columns of the array being the blocks of the large relation $A$. 
\vspace{2ex}

\begin{center}
\begin{scriptsize}
\begin{tabular}{cccc}
&1&10&19\\
&2&11&20\\
&3&12&21\\
&4&13&22\\
&5&14&23\\
&6&15&24\\
&7&16&25\\
&8&17&26\\
&9&18&27
\end{tabular}
\end{scriptsize}
\end{center}
\vspace{1ex}

\begin{defn}
Atoms $aA$ and $aB$ with the same first spot are near to one another if $A\cap B$ has one block of 9 elements, 2 blocks of 6 elements, and 2 blocks of 3 elements. 
\end{defn}

Note that $aA$ and $aB$ are near if $aB$ is obtained by picking three elements $x_1,x_2,x_3$ in a block $A_1$ of $A$, picking a second block $A_2$ of $A$, finding the three elements $y_1,y_2,y_3$ in $A_2$ with $x_i$ and $y_i$ related by $a$, and then constructing $B$ by {\em switching} $x_1,x_2,x_3$ and $y_1,y_2,y_3$. More precisely, let $B_1=A_1-\{x_1,x_2,x_3\}\cup\{y_1,y_2,y_3\}$, $B_2=A_2-\{y_1,y_2,y_3\}\cup\{x_1,x_2,x_3\}$, and $B_3=A_3$. Below is an example of two atom near to one another. The bullet marks indicate the rows of the elements involved in the {\em switching}. 
\vspace{2ex}

\begin{center}
\begin{scriptsize}
\begin{tabular}{cccc}
\bt &1&10&19\\
\bt &2&11&20\\
\bt &3&12&21\\
&4&13&22\\
&5&14&23\\
&6&15&24\\
&7&16&25\\
&8&17&26\\
&9&18&27
\end{tabular}
\quad \quad
\begin{tabular}{cccc}
&10&1&19\\
&11&2&20\\
&12&3&21\\
&4&13&22\\
&5&14&23\\
&6&15&24\\
&7&16&25\\
&8&17&26\\
&9&18&27
\end{tabular}
\end{scriptsize}
\end{center}
\vspace{1ex}

\begin{lemma}
If $aA$ and $aB$ are near, there is a small $d$ with $aA$ and $aB$ in $\mathfrak{X}(a,d)$. 
\end{lemma}

\begin{proof}
Up to relabeling of elements of $X$, the situation shown above is typical. Let $d$ be the small relation whose blocks are obtained by splitting each column into three batches of three so that the swapped elements form two blocks. For instance, d might be $\{1,2,3\}$, $\{4,5,6\}$, $\{7,8,9\}$, $\{10,11,12\}$, $\{13,14,15\}$, $\{16,17,18\}$, $\{19,20,21\}$, $\{22,23,24\}$, $\{25, 26,27\}$. 
\end{proof}

\begin{prop}
If $aA$ and $aB$ are atoms with the same first spot, then their images under an automorphism $\Phi$ of $\ts{Fact}~X$ also have the same first spots. 
\label{firstspots}
\end{prop}

\begin{proof}
Note first that any $\mathfrak{X}(a,d)$ and $\mathcal{Y}(a,d)$ are two sets of 36 atoms each with $\mathfrak{X}(a,d)\perp \mathcal{Y}(a,d)$. Therefore the same is true of the images of these sets under $\Phi$, so Proposition~\ref{main} says the images of these sets are of the form $\mathfrak{X}(a',d')$ and $\mathcal{Y}(a',d')$ for some $a',d'$. So if $aA$ and $aA'$ are two atoms that belong to some $\mathfrak{X}(a,d)$, then the images of $aA$ and $aA'$ under $\Phi$ must have the same first spots. In particular, if $aA$ and $aA'$ are near, their images have the same first spots. 

\begin{claim}
If $aA$ and $aB$ are atoms and $B$ is formed from $A$ by swapping two elements that lie in the same block of $a$, then the images of $aA$ and $aB$ under $\Phi$ have the same first spots. 
\end{claim}

\begin{proof}[Proof of Claim:] Up to relabeling the elements of $X$, we may assume $aA$ is the factor pair discussed above and that $B$ is formed from $A$ by swapping the elements $1,10$. In the figure below, we have six factor pairs $aA_0, \ldots, aA_5$ with $A_0=A$ and $A_5 = B$.  
\vspace{2ex}

\begin{center}
\begin{scriptsize}
\begin{tabular}{rrr}
1&10&19\\
2&11&20\\
3&12&21\\
4&13&22\\
5&14&23\\
6&15&24\\
\bt\, 7&16&25\\
\bt\, 8&17&26\\
\bt\, 9&18&27
\end{tabular}
\hspace{1ex}
\begin{tabular}{rrr}
1&10&19\\
2&11&20\\
3&12&21\\
\bt\, 4&13&22\\
\bt\, 5&14&23\\
\bt\, 6&15&24\\
16&7&25\\
17&8&26\\
18&9&27
\end{tabular}
\hspace{1ex}
\begin{tabular}{rrr}
\bt\, 1&10&19\\
\bt\, 2&11&20\\
3&12&21\\
13&4&22\\
14&5&23\\
15&6&24\\
\bt\, 16&7&25\\
17&8&26\\
18&9&27
\end{tabular}
\hspace{1ex}
\begin{tabular}{rrr}
10&1&19\\
11&2&20\\
3&12&21\\
13&4&22\\
\bt\, 14&5&23\\
\bt\, 15&6&24\\
7&16&25\\
\bt\, 17&8&26\\
18&9&27
\end{tabular}
\hspace{1ex}
\begin{tabular}{rrr}
10&1&19\\
\bt\, 11&2&20\\
3&12&21\\
\bt\, 13&4&22\\
5&14&23\\
6&15&24\\
7&16&25\\
8&17&26\\
\bt\, 18&9&27
\end{tabular}
\hspace{1ex}
\begin{tabular}{rrr}
10&1&19\\
2&11&20\\
3&12&21\\
4&13&22\\
5&14&23\\
6&15&24\\
7&16&25\\
8&17&26\\
9&18&27
\end{tabular}
\end{scriptsize}
\end{center}
\vspace{2ex}

\noindent Note $aA_i$ and $aA_{i+1}$ are near for each $i=0,\ldots,4$, so their images have the same first spot. Thus the images of $aA$ and $aB$ have the same first spot.
\end{proof}

To conclude the proof of Proposition~\ref{firstspots}, we note that for any atoms $aA$ and $aB$ having the same first spots, that $B$ can be formed by repeatedly swapping two elements in $A$ that belong to the same block of $a$. So the above claim shows the images of $aA$ and $aB$ have the same first spots. 
\end{proof}

We next consider matters for images of second spots of atoms. 

\begin{prop}
Suppose $a,b$ permute and $a\cap b=\Delta$. Let $C=a\circ b$ and define 
\[\mathcal{Z}(a,b) = \{cC: a,b,c \mbox{ is a factor triple}\}\]
Then $\mathcal{Z}(a,b)$ is those atoms $cC$ where there are an $aA\in\mathfrak{X}(a,b)$ and $bB\in\mathcal{Y}(a,b)$ with $aA, bB, cC$ pairwise orthogonal. Further, there are $32\cdot 32$ elements in $\mathcal{Z}(a,b)$.  
\end{prop}

\begin{proof}
Each such orthogonal triple $a,b,c$ gives atoms $a(b\circ c)$ $b(a\circ c)$ and $c(a\circ b)$. These are pairwise orthogonal atoms with the first and second belonging to $\mathfrak{X}(a,b)$ and $\mathcal{Y}(a,b)$ respectively, and the third to $\mathcal{Z}(a,b)$. Any two of these atoms determine the third, and our result follows. 
\end{proof}

\begin{defn}
Call $\mathfrak{X}$ a slab if it is a set of 36 atoms and there is another set $\mathcal{Y}$ of 36 atoms with $\mathfrak{X}\perp\mathcal{Y}$. Call $\mathfrak{X},\mathcal{Y}, \mathcal{Z}$ a triple if $\mathfrak{X}\perp\mathcal{Y}$ and the members of $\mathcal{Z}$ are exactly the atoms making up blocks with atoms from $\mathfrak{X}$ and $\mathcal{Y}$. 
\end{defn}

We have seen in Proposition~\ref{main} that each slab $\mathfrak{X}$ is of the form $\mathfrak{X}(a,b)$ for some small $a,b$. These $a,b$ are uniquely determined as $a$ is the first component of members of $\mathfrak{X}$ and $b$ is the intersection of the second components of $\mathfrak{X}$. So a slab $\mathfrak{X}$ determines uniquely its companion $\mathcal{Y}$ and the resulting set $\mathcal{Z}$ of atoms making up blocks with elements of $\mathfrak{X}$ and $\mathcal{Y}$. So $\mathfrak{X}$ determines the triple $\mathfrak{X},\mathcal{Y},\mathcal{Z}$. 

\begin{prop}
If $\mathfrak{X},\mathcal{Y},\mathcal{Z}$ is a triple, then any two atoms in $\mathcal{Z}$ have the same second spot, and the images of these atoms under an automorphism $\Phi$ of $\ts{Fact}~X$ have the same second spot. 
\label{gecko}
\end{prop}

\begin{proof}
From the comments above, this triple is $\mathfrak{X}(a,b), \mathcal{Y}(a,b), \mathcal{Z}(a,b)$ for some $a,b$. The second spots of members of $\mathcal{Z}$ must then be $a\circ b$, so all members of $\mathcal{Z}$ have the same second spot. The images of $\mathfrak{X},\mathcal{Y},\mathcal{Z}$ under $\Phi$ form a triple. As members of $\mathfrak{X}$ have the same first spots, by Proposition~\ref{firstspots} their images under $\Phi$ have the same first spot $a'$, and the images of members of $\mathcal{Y}$ have the same first spots $b'$. So the image of $\mathfrak{X}$ is $\mathfrak{X}(a',b')$, the image of $\mathcal{Y}$ is $\mathcal{Y}(a',b')$, hence the image of $\mathcal{Z}$ is $\mathcal{Z}(a',b')$. Therefore the images of members of $\mathcal{Z}$ all have the same second spot $a'\circ b'$. 
\end{proof}

Earlier, we defined nearness for atoms with the same first spot. We now define a notion of nearness for atoms with the same second spot. Here we use $\oplus$ for the symmetric difference of sets. 

\begin{defn}
Atoms $aA$ and $bA$ with the same second spot are near if $b$ is formed from $a$ as follows. Find by 6 elements $x_1,\ldots,x_6$ in a block $A_1$ of $A$ and let $a_1,\ldots,a_6$ be blocks of $a$ with $x_i\in a_i$. Let $b$ have blocks $b_1,\ldots,b_9$ such that
for $i=1,3,5$ $b_i=a_i\oplus\{a_1,a_{i+1}\}$ and $b_{i+1}=a_{i+1}\oplus\{a_i,a_{i+1}\}$, and $b_i=a_i$ for $i=7,8,9$. 
\end{defn}

So $aA$ and $bA$ are near if $b$ is formed from $a$ by making three non-overlapping swaps of elements all in one block of $A$. A typical example is shown below where we use like symbols at left to indicate the the rows of the pairs of swaps. 

\vspace{2ex}
\begin{center}
\begin{scriptsize}
\begin{tabular}{rrr}
$\triangle$ 1&10&19\\
$\triangle$ 2&11&20\\
$\star$\, 3&12&21\\
$\star$\, 4&13&22\\
\bt\, 5&14&23\\
\bt\, 6&15&24\\
7&16&25\\
8&17&26\\
9&18&27
\end{tabular}
\hspace{5ex}
\begin{tabular}{rrr}
1&10&20\\
2&11&19\\
3&12&22\\
4&13&21\\
5&14&24\\
6&15&23\\
7&16&25\\
8&17&26\\
9&18&27
\end{tabular}
\end{scriptsize}
\end{center}
\vspace{1ex}

\begin{lemma}
If $aA$ and $bA$ are near, there is a triple $\mathfrak{X},\mathcal{Y},\mathcal{Z}$ with $aA, bA\in \mathcal{Z}$. 
\label{emu}
\end{lemma}

\begin{proof}
Reverting to our earlier method of representing $X$, assume elements of $X$ are triples such as 012 arranged on the $x,y,z$-axes. We assume without loss of generality that $A$, $a$, and $b$ are as follows. The blocks of $A$ are the planes $x=0,1,2$ and $A$ is at left below; the blocks of $a$ are lines parallel to the $x$-axis shown in the middle below; and the small relation $b$ is shown at right below. Any near $aA$ and $bA$ can be realized in this way for some enumeration of $X$. 

\vspace{2ex}
\setlength{\unitlength}{.014in}
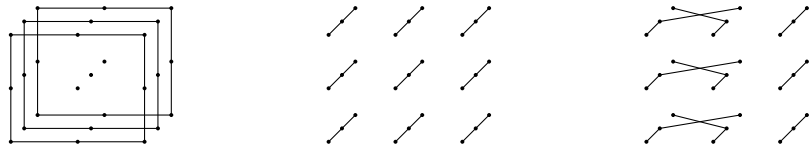
\begin{center}
\begin{picture}(60,50)(0,0)
\multiput(0,0)(25,0){3}{
\multiput(0,0)(0,20){3}{\circle*{1.5}}
\multiput(5,5)(0,20){3}{\circle*{1.5}}
\multiput(10,10)(0,20){3}{\circle*{1.5}}
}
\multiput(0,0)(5,5){3}{
\put(0,0){\line(1,0){50}}
\put(0,0){\line(0,1){40}}
\put(0,40){\line(1,0){50}}
\put(50,0){\line(0,1){40}}
}
\end{picture}
\hspace{10ex}
\begin{picture}(60,50)(0,0)
\multiput(0,0)(25,0){3}{
\multiput(0,0)(0,20){3}{\circle*{1.5}}
\multiput(5,5)(0,20){3}{\circle*{1.5}}
\multiput(10,10)(0,20){3}{\circle*{1.5}}
}
\multiput(0,0)(0,20){3}{
\put(0,0){\line(1,1){10}}
\put(25,0){\line(1,1){10}}
\put(50,0){\line(1,1){10}}
}
\end{picture}
\hspace{10ex}
\begin{picture}(60,50)(0,0)
\multiput(0,0)(25,0){3}{
\multiput(0,0)(0,20){3}{\circle*{1.5}}
\multiput(5,5)(0,20){3}{\circle*{1.5}}
\multiput(10,10)(0,20){3}{\circle*{1.5}}
}
\multiput(0,0)(0,20){3}{
\put(0,0){\line(1,1){5}}
\put(25,0){\line(1,1){5}}
\put(50,0){\line(1,1){10}}
\put(5,5){\line(6,1){30}}
\put(10,10){\line(4,-1){20}}
}
\end{picture}
\end{center}

Then consider the small relation $p$ whose blocks are lines parallel to the $y$-axis, and the small relation $q$ whose blocks are lines parallel to the $z$-axis. We have shown $p,q$ below at left along with the relations $a,b$ described above. 

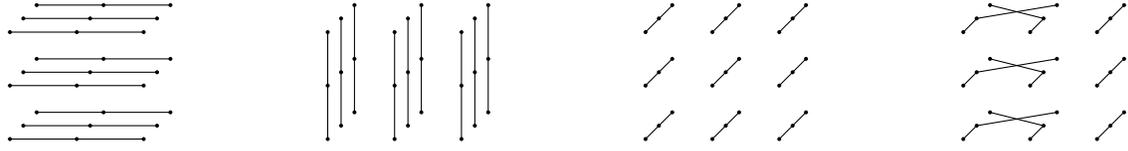
\begin{figure}[h]
\setlength{\unitlength}{.014in}
\begin{center}
\begin{picture}(60,50)(0,0)
\multiput(0,0)(25,0){3}{
\multiput(0,0)(0,20){3}{\circle*{1.5}}
\multiput(5,5)(0,20){3}{\circle*{1.5}}
\multiput(10,10)(0,20){3}{\circle*{1.5}}
}
\multiput(0,0)(5,5){3}{
\put(0,0){\line(1,0){50}}
\put(0,20){\line(1,0){50}}
\put(0,40){\line(1,0){50}}
}
\end{picture}
\hspace{10ex}
\begin{picture}(60,50)(0,0)
\multiput(0,0)(25,0){3}{
\multiput(0,0)(0,20){3}{\circle*{1.5}}
\multiput(5,5)(0,20){3}{\circle*{1.5}}
\multiput(10,10)(0,20){3}{\circle*{1.5}}
}
\multiput(0,0)(5,5){3}{
\put(0,0){\line(0,1){40}}
\put(25,0){\line(0,1){40}}
\put(50,0){\line(0,1){40}}
}
\end{picture}
\hspace{10ex}
\begin{picture}(60,50)(0,0)
\multiput(0,0)(25,0){3}{
\multiput(0,0)(0,20){3}{\circle*{1.5}}
\multiput(5,5)(0,20){3}{\circle*{1.5}}
\multiput(10,10)(0,20){3}{\circle*{1.5}}
}
\multiput(0,0)(0,20){3}{
\put(0,0){\line(1,1){10}}
\put(25,0){\line(1,1){10}}
\put(50,0){\line(1,1){10}}
}
\end{picture}
\hspace{10ex}
\begin{picture}(60,50)(0,0)
\multiput(0,0)(25,0){3}{
\multiput(0,0)(0,20){3}{\circle*{1.5}}
\multiput(5,5)(0,20){3}{\circle*{1.5}}
\multiput(10,10)(0,20){3}{\circle*{1.5}}
}
\multiput(0,0)(0,20){3}{
\put(0,0){\line(1,1){5}}
\put(25,0){\line(1,1){5}}
\put(50,0){\line(1,1){10}}
\put(5,5){\line(6,1){30}}
\put(10,10){\line(4,-1){20}}
}
\end{picture}
\end{center}
\end{figure}

\noindent Clearly $a,p,q$ are an orthogonal triple. To see $b,p,q$ are an orthogonal triple, it is enough to see $b,p,q$ pairwise permute, and $b\cap(p\circ q)=p\cap(b\circ q)=q\cap(b\circ p)=\Delta$. Thus $aA$ and $bA$ belong to $\mathcal{Z}(p,q)$. 
\end{proof}

\begin{prop}
If $aA$ and $bA$ are atoms with the same second spot, then their images under an automorphism $\Phi$ of $\ts{Fact}~X$ have the same second spot. 
\label{secondspots}
\end{prop}

\begin{proof}
We first show that if $aA$ and $cA$ are such that $a$ and $c$ agree except for having two elements in the same block of $A$ interchanged, then $\Phi(aA)$ and $\Phi(cA)$ have the same second spot. Consider the situation below that represents four atoms $aA$, $uA$, $vA$, $cA$ where $aA$ and $cA$ are such that $a,c$ agree except for two elements swapped. Each pair of elements in this sequence are near atoms with the same second spot, so by Lemma~\ref{emu} and Proposition~\ref{gecko} their images have the same second spot. Thus $\Phi(aA)$ and $\Phi(cA)$ have the same second spots. Up to rearrangement of elements of $X$, this argument is general. 

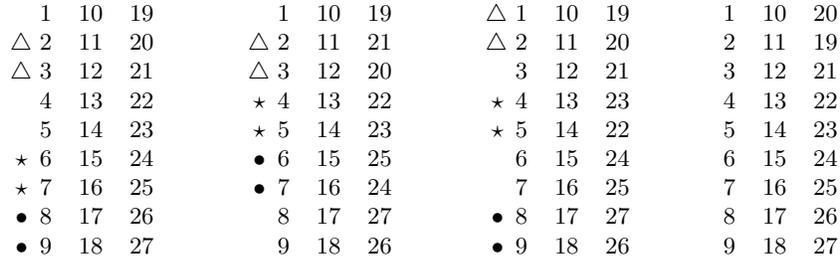
\vspace{1ex}
\begin{center}
\begin{scriptsize}
\begin{tabular}{rrr}
1&10&19\\
$\triangle$ 2&11&20\\
$\triangle$ 3&12&21\\
4&13&22\\
5&14&23\\
$\star$\, 6&15&24\\
$\star$\, 7&16&25\\
\bt\, 8&17&26\\
\bt\, 9&18&27
\end{tabular}
\hspace{5ex}
\begin{tabular}{rrr}
1&10&19\\
$\triangle$ 2&11&21\\
$\triangle$ 3&12&20\\
$\star$\, 4&13&22\\
$\star$\, 5&14&23\\
\bt\, 6&15&25\\
\bt\, 7&16&24\\
8&17&27\\
9&18&26
\end{tabular}
\hspace{5ex}
\begin{tabular}{rrr}
$\triangle$ 1&10&19\\
$\triangle$ 2&11&20\\
3&12&21\\
$\star$\, 4&13&23\\
$\star$\, 5&14&22\\
6&15&24\\
7&16&25\\
\bt\, 8&17&27\\
\bt\, 9&18&26
\end{tabular}
\hspace{5ex}
\begin{tabular}{rrr}
1&10&20\\
2&11&19\\
3&12&21\\
4&13&22\\
5&14&23\\
6&15&24\\
7&16&25\\
8&17&26\\
9&18&27
\end{tabular}
\end{scriptsize}
\end{center}
\vspace{1ex}

To complete the proof we have only to note that for any $aA$ and $bA$ with the same second spot, we can transform $a$ into $b$ through a series of moves, each swapping two elements belonging to the same block of $A$. 
\end{proof}

Propositions~\ref{firstspots} and \ref{secondspots} allow the following definition. 

\begin{defn}
For an automorphism $\Phi$ of $\ts{Fact}~X$, define endomorphisms $\Phi_{\ts{S}}$ of the small relations and $\Phi_{\ts{L}}$ of the large relations as follows. 
\vspace{1ex}
\begin{eqnarray*}
\Phi_{\ts{S}}\, a &=& \mbox{the first spot of any $\Phi(aA)$ where $aA$ is an atom}\\
\Phi_{\ts{L}}\, A &=& \mbox{the second spot of any $\Phi(aA)$ where $aA$ is an atom}
\end{eqnarray*}
\end{defn}

The following is easily verified form the definition. 

\pagebreak[3]

\begin{lemma}
For $\Phi,\Psi$ automorphisms of $\ts{Fact}~X$ we have 
\vspace{1ex}
\begin{enumerate}
\item $(\Psi\circ\Phi)_{\ts{S}}=\Phi_{\ts{S}}\circ\Phi_{\ts{S}}$
\item $(\Psi\circ\Phi)_{\ts{L}}=\Phi_{\ts{L}}\circ\Phi_{\ts{L}}$
\end{enumerate}
\vspace{1ex}
\noindent Further, if $\Phi$ is the identity map on $\ts{Fact}~X$, both $\Phi_{\ts{S}}$ and $\phi_{\ts{L}}$ are identity maps. 
\label{hhh}
\end{lemma}

It follows that $\Phi_{\ts{S}}$ is a permutation of the set of small relations of $X$, $\Phi_{\ts{L}}$ is a permutation of the large relations, and that the obvious maps $\ts{S}$ and $\ts{L}$ from the automorphism group of $\ts{Fact}~X$ to the permutation groups of the small and large relations are group homomorphisms. 

\begin{thm}
There is a group embedding $\ts{R}$ of the automorphism group of $\ts{Fact}~X$ into the group of order-automorphisms of the poset $\ts{Req}~X$ taking $\Phi$ to $\Phi_{\ts{R}}$ where 

\[\Phi_{\ts{R}}\, x =
\left\{
	\begin{array}{ll}
		\Phi_{\ts{S}}\,x  & \mbox{if } x \mbox{ is small } \\
		\Phi_{\ts{L}}\,x & \mbox{if $x$ is large}
	\end{array}
\right.\]
\label{c}
\end{thm}

\begin{proof}
From Lemma~\ref{hhh} it follows that each $\Phi_{\ts{R}}$ is a permutation of the regular equivalence relations, and that $\ts{R}$ is a group homomorphism from the automorphism group of $\ts{Fact}~X$ to the group of permutations of $\ts{Req}~X$. From the definition, it is trivial that $\Phi_{\ts{R}}$ is the identity iff $\Phi$ is the identity, so $\ts{R}$ is a group embedding. It remains only to show that each $\Phi_{\ts{R}}$ is an order-embedding. For this, we know $\ts{R}$ is a group homomorphism, so it is enough to show $\Phi_{\ts{R}}$ is order-preserving. 

Suppose $a$ is small and $A$ is large with $a\leq A$. Then there is a small $b$ with $a\cap b=\Delta$, $a$ permuting with $b$, and $a\circ b=A$. So there is a triple $\mathfrak{X},\mathcal{Y},\mathcal{Z}$ with the first spots of members of $\mathfrak{X}$ being $a$, the first spots of members of $\mathcal{Y}$  being $b$, and the second spots of members of $\mathcal{Z}$ being $A$. As we have seen, the images of $\mathfrak{X},\mathcal{Y},\mathcal{Z}$ under $\Phi$ are a triple with the first spots of members of the image of $\mathfrak{X}$ being $\Phi_{\ts{S}}\,a$, the first spots of members of the image of $\mathcal{Y}$ being $\Phi_{\ts{S}}\,b$, and the second spots of members of the image of $\mathcal{Z}$ being $\Phi_{\ts{L}}\,A$. This implies $\Phi_{\ts{S}}\,a\leq\Phi_{\ts{L}}\,A$, and shows $\Phi_{\ts{R}}$ is order-preserving. 
\end{proof}

\section{The case of a 27-element set --- the second half}

Here we complete the proof that for $X$ a 27-element set, the map $\Gamma$ gives an isomorphism from the permutation group of $X$ to the group of automorphisms of $\ts{Fact}~X$. In particular, we show the automorphisms of the poset $\ts{Req}~X$ correspond to permutations of the set $X$. 

\begin{lemma}
For small $a,b$, if $\textsc{Trcl}(a\cup b)$ has 7 blocks of 3 and 1 block of 6, then $a,b$ have 70 large upper bounds. 
\label{70ubs}
\end{lemma}

\begin{proof}
Note that any equivalence relation containing $a$ joins together blocks of $a$, and this includes $\textsc{Trcl}(a\cup b)$. As $\textsc{Trcl}(a\cup b)$ has 7 blocks of 3 and 1 block of 6, the large relations containing this transitive closure are formed by picking one of the blocks of 3 to match with the block of 6, and there are 7 ways to do this, then splitting the remaining 6 blocks of 3 into two batches of 3, and there are 10 ways to do this. 
\end{proof}

We extend this notion to other pairs of elements. Here we write (1-6, 7-3) = 70 to mean if $\textsc{Trcl}(a\cup b)$ has 1 block of 6, and 7 blocks of 3, then $a,b$ have 70 large upper bounds in common. We will not prove the following result, but will use it. Its proof is a matter of counting along the lines above. 

\begin{lemma}
For small relations $a,b$, we have the following. 
\vspace{1ex}

{\em 
\begin{tabular}{lllllll}
(9-3) = 280   && (1-6,7-3) = 70   && (1-9,6-3) = 10   && (2-6,5-3) = 20    \\
(3-6,3-3) = 6   &&  (4-6,1-3) = 0  && (1-9,1-6,4-3) = 4   && (1-9,2-6,2-3) = 2    \\
(1-9,3-6) = 0   && (2-9,3-3) = 1   && (2-9,1-6,1-3) = 1   && (3-9) = 1    
\end{tabular}
}
\label{countubs}
\end{lemma}

\begin{defn}
Suppose $a$ is small and $a_1,\ldots,a_9$ are the blocks of $a$. Define
\[\mathfrak{X}(a:a_i,a_j) = \{b:b\mbox{ is small and $n\neq i,j\Rightarrow a_n$ is a block of $b$}\}.\]
\end{defn}

Our next aim is be able to abstractly recognize such sets of small relations in the poset of regular equivalence relations. This will show such sets are mapped in a well-behaved manner by automorphisms of $\ts{Req}~X$.

\begin{lemma}
$\mathfrak{X}(a:a_i,a_j)$ has 10 elements, and any two distinct elements of this set have 70 large upper bounds in common. Conversely, any set $\mathfrak{X}$ of 10 small relations where any two have 70 large upper bounds is of the form $\mathfrak{X}(a:a_i,a_j)$ where $a$ can be chosen to be any member of $\mathfrak{X}$ and $a_i\cup a_j$ is the same no matter the $a$ chosen. 
\label{2collapse}
\end{lemma}

\begin{proof}
The elements of $\mathfrak{X}(a:a_i,a_j)$ are those small $b$ that share at least $7$ blocks with $a$. These are the $b$'s formed by taking these 7 blocks of $a$, then splitting up $a_i\cup a_j$ into two groups of 3, where the order of these groups, and in these groups, doesn't matter. So there are 10 such $b$, one of which will be $a$. Lemma~\ref{70ubs} shows that any two distinct members of this group have 70 large upper bounds in common. 

Conversely, suppose $\mathfrak{X}$ is such a set. Pick any element in this, say $a$. For another element $b$ in this set to have 70 large upper bounds in common with $a$, it must be that $\textsc{Trcl}(a\cup b)$ has 1 block of 6 and 7 blocks of 3. So $a,b$ have 7 blocks in common. Take a third element $c$ in $\mathfrak{X}$. We must have that any two of $a,b,c$ share 7 blocks. Say $a,b$ share the blocks $a_3,\ldots,a_9$ of $a$. Then as $a\neq b$ we have neither $a_1,a_2$ is a block of $b$. So the blocks of $b$ are $b_1,b_2,a_3,\ldots,a_9$ where neither $b_1,b_2$ equals $a_1,a_2$. Suppose $a,c$ do not share the same 7 blocks $a_3,\ldots,a_9$. Say $a_3$ is not a block of $c$. Then for $c$ to share 7 blocks with $a$ we must have one of $a_1,a_2$ is a block of $c$. Say $a_1$ is a block of $c$. Then neither $b_1,b_2$ can be a block of $c$, since both contain some, but not all of $a_1$. Also $a_3$ is not a block of $c$, but is a block of $b$. So $c$ cannot share 7 blocks with $b$. Thus all members of $\mathfrak{X}$ have the same 7 blocks of $a$ in common. So $\mathfrak{X}=\mathfrak{X}(a:a_1,a_2)$. If we had chosen a different element from $\mathfrak{X}$ to begin, say $b$, the blocks $b_i,b_j$ would have been different, but $b_i\cup b_j$ would equal $a_1\cup a_2$. 
\end{proof}

\begin{defn}
A set $\mathfrak{X}$ of small relations is called a collapse if it satisfies the conditions of Lemma~\ref{2collapse}. We say $a,b$ are neighbours and write $a\sim b$ if they both belong to a collapse. 
\end{defn}

\begin{defn}
Two collapses $\mathfrak{X}$ and $\mathcal{Y}$ are said to share a block if $\mathfrak{X}\cap\mathcal{Y}\neq\emptyset$ and no  $a,b$ in $\mathfrak{X}\cup\mathcal{Y}$ have exactly 20 common large upper bounds. 
\end{defn} 

\begin{lemma} 
Let $a$ be a small relation with blocks $a_1,\ldots,a_9$. Then the collapses $\mathfrak{X}(a:a_i,a_j)$ and $\mathfrak{X}(a:a_m,a_n)$ share a block iff one of $i,j$ equals one of $m,n$. 
\label{shareblock}
\end{lemma}

\begin{proof}
Suppose these collapses share a block. If $a_i,a_j,a_m,a_n$ are all different, then there would be elements $b\in\mathfrak{X}$ and $c\in\mathcal{Y}$ with $\textsc{Trcl}(b\cup c)$ having 2 blocks of 6 and 5 blocks of 3, hence by Lemma~\ref{countubs}, the pair $b,c$ would have 20 upper bounds. So one of $a_i,a_j$ equals one of $a_m,a_n$. Conversely, if $i=m$, then for any $b,c$ in $\mathfrak{X}(a:a_i,a_j)\cup \mathfrak{X}(a:a_i,a_n)$ we have $\textsc{Trcl}(b\cup c)$ has either 1 block of 9 and 6 blocks of 3 (if one is from each set), or has 1 block of 6 and 7 blocks of 3 (if both are from the same. In either case, by Lemma~\ref{countubs} they do not have exactly 20 upper bounds. 
\end{proof}

\begin{lemma}
Each small relation $a$ is in exactly 36 collapses. 
\end{lemma}

\begin{proof}
If $a$ is in a collapse $\mathfrak{X}$, then by Lemma~\ref{2collapse}, we have $\mathfrak{X}=\mathfrak{X}(a:a_i,a_j)$ for two blocks $a_i,a_j$ of $a$. There are 36 ways to choose these two blocks, since their order does not matter. 
\end{proof}

\begin{defn}
A 3-element subset $\alpha$ of $X$ is called a small block. For such $\alpha$, set 
\[\mathfrak{X}_\alpha = \{a:\alpha\mbox{ is a block of }a\}.\]
\end{defn}

\begin{lemma}
Let $\mathfrak{X}$ be a set of $24!/(8!)(3!)^8$ small relations with $\sim$ the restriction of the neighbour relation to $\mathfrak{X}$, and suppose the following hold for all $a,b,c\in \mathfrak{X}$.
\vspace{1ex}

\begin{enumerate}
\item 28 of the collapses containing $a$ are subsets of $\mathfrak{X}$. 
\item 8 of the collapses containing $a$ intersect $\mathfrak{X}$ only in $\{a\}$. 
\item Any two collapses of $a$ that intersect $\mathfrak{X}$ only in $\{a\}$ share a block. 
\item If $a\sim b,c$ and $b,c$ have 20 upper bounds, there is $d\neq a$ in $\mathfrak{X}$ with $b,c\sim d$.
\end{enumerate}
\vspace{1ex}

\noindent Then $\mathfrak{X}=\mathfrak{X}_\alpha$ for some small block $\alpha$. Conversely, each $\mathfrak{X}_\alpha$ satisfies these properties.  
\label{sue}
\end{lemma}

\begin{proof}
We first show $\mathfrak{X}_\alpha$ satisfies these properties. That it has the indicated number of elements is routine. 
Suppose $a\in\mathfrak{X}_\alpha$, and that the blocks of $a$ are $a_1,\ldots,a_9$ with $a_1=\alpha$. By Lemma~\ref{2collapse} the collapses containing $a$ are exactly the $\mathfrak{X}(a:a_i,a_j)$. If either $i,j$ equals $1$, then the only element of this collapse having $\alpha$ as a block is $a$, and if $i,j\neq 1$, then all elements of this collapse have $\alpha$ as a block, so this collapse is contained in $\mathfrak{X}_\alpha$. So there are 28 collapses containing $a$ that are contained in $\mathfrak{X}_\alpha$ and 8 that intersect $\mathfrak{X}_\alpha$ only in $\{a\}$. Any two collapses that intersect $\mathfrak{X}_\alpha$ only in $\{a\}$ are of the form $\mathfrak{X}(a:a_1,a_i)$ and $\mathfrak{X}(a:a_1,a_j)$, so by Lemma~\ref{shareblock} they share a block. 
For the final condition, suppose $b,c\in\mathfrak{X}_\alpha$ with $a\sim b,c$ and that $b,c$ have 20 large upper bounds. This means $b$ is in a collapse of $a$, and this collapse must be of the form  $\mathfrak{X}(a:a_i,a_i)$ with $i,j\neq 1$, and similarly $c$ is in $\mathfrak{X}(a:a_m,a_n)$ with $m,n\neq 1$. Since $b,c$ have 20 large upper bounds, $\textsc{Trcl}(b\cup c)$ has 2 blocks of 6 and 5 blocks of 3, and the two blocks of 6 are $a_i\cup a_j$ and $a_m\cup a_n$. So $i,j,m,n$ are all distinct. Form $d$ to behave like $b$ on $a_i\cup a_j$ and like $c$ on $a_m\cup a_n$, and to agree with $a$ elsewhere. Then $d$ is in the collapse $\mathfrak{X}(b:a_m,a_n)$ and also in the collapse $\mathfrak{X}(c:a_i,a_j)$, so $b,c\sim d$. As $i,j,m,n\neq 1$, we have $d\in \mathfrak{X}$. So $\mathfrak{X}_\alpha$ satisfies these conditions. 

For the forward direction, suppose that $\mathfrak{X}$ is a set with the indicated number of small relations that satisfies these conditions. 

\begin{claim}
For $a\in\mathfrak{X}$ with blocks $a_1,\ldots,a_9$, there is a block $a_k$, that we denote $\alpha(a)$, with $\mathfrak{X}(a:a_i,a_j)\subseteq \mathfrak{X}$  if $i,j\neq k$ and $\mathfrak{X}(a:a_i,a_j)\cap\mathfrak{X}=\{a\}$ if either $i,j=k$.
\end{claim}

\begin{proof}[Proof of Claim:] 
There are 8 collapses $\mathfrak{X}(a:a_i,a_j)$ that intersect $\mathfrak{X}$ only in $\{a\}$, and any two of these share a block. So they must share the same block, some $a_k$, and these must be all collapses of $a$ using this block. 
\end{proof}

\begin{claim}
If $a,b\in\mathfrak{X}$ and $a\sim b$, then $\alpha(a)=\alpha(b)$. 
\label{llk}
\end{claim}

\begin{proof}[Proof of Claim:]
Suppose $a$ has blocks $a_1,\ldots,a_9$ and $\alpha(a)=a_1$. As $b\in\mathfrak{X}$ and $a\sim b$ we have $b\in\mathfrak{X}(a:a_i,a_j)$ for some $i,j\neq 1$. We assume $b\in\mathfrak{X}(a:a_2,a_3)$. Let the blocks of $b$ be $b_1,\ldots,b_9$ with the numbering chosen so that $b_1=a_1$, $b_2\cup b_3=a_2\cup a_3$ and $b_i=a_i$ for $i=4,\ldots,9$. We must show $\alpha(b)=b_1$. 

We know $\alpha(b)\neq b_2,b_3$ since $a\in\mathfrak{X}(b:b_2,b_3)$ showing that this collapse does not intersect $\mathfrak{X}$ only in $\{b\}$. We show $\alpha(b)\neq b_4,\ldots,b_9$, hence $\alpha(b)=b_1$ as required. We provide the argument to show $\alpha(b)\neq b_4$, the others follow by symmetry. 

Choose any element $c$ in $\mathfrak{X}(a:a_4,a_5)$ distinct from $a$. Let the blocks of $c$ be numbered $c_1,\ldots,c_9$ with $c_i=a_i$ for $i\neq 4,5$ and with $c_4\cup c_5=a_4\cup a_5$. Then $\textsc{Trcl}(b\cup c)$ has 2 blocks of 6, namely $a_2\cup a_3$ and $a_4\cup a_5$, and 5 blocks of 3, the blocks $a_1,a_6,\ldots,a_9$. So $b,c$ have 20 large upper bounds. Then by condition~(4) there is $d\in \mathfrak{X}$ distinct from $a$ with $b,c\sim d$. Since $d$ is in a collapse of $b$ and a collapse of $c$ and differs from $a$, the only possibility for $d$ is to have $d$ agree with $a$ (and with $b,c$) on $a_1, a_6,\ldots,a_9$ and for $d$ to be such that $d$ agrees with $b$ on $a_2\cup a_3$ and with $c$ on $a_4\cup a_5$. As $b\sim d$ we must have $d\in\mathfrak{X}(b:b_i,b_j)$, and clearly $i,j$ must be $4,5$. This shows $\mathfrak{X}(b:b_4,b_5)$ does not intersect $\mathfrak{X}$ only in $b$, hence $\alpha(b)\neq b_4$. 
\end{proof}

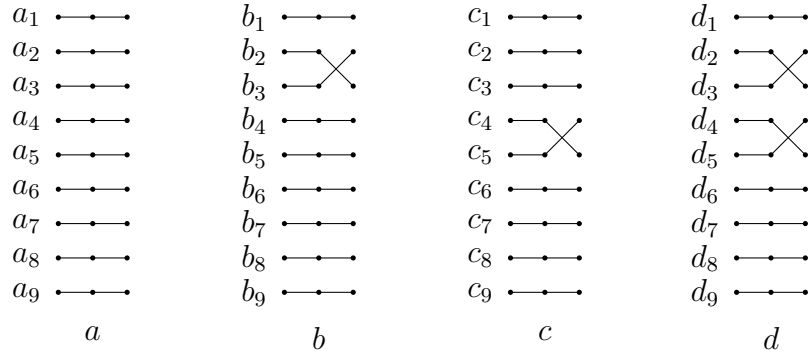
\begin{figure}[h]
\setlength{\unitlength}{.018in}
\begin{center}
\begin{picture}(20,90)(0,-10)
\multiput(0,0)(0,10){9}{
\multiput(0,0)(10,0){3}{\circle*{1.5}}
\put(0,0){\line(1,0){20}}
}
\put(-5,80){\makebox(0,0)[r]{$a_1$}}
\put(-5,70){\makebox(0,0)[r]{$a_2$}}
\put(-5,60){\makebox(0,0)[r]{$a_3$}}
\put(-5,50){\makebox(0,0)[r]{$a_4$}}
\put(-5,40){\makebox(0,0)[r]{$a_5$}}
\put(-5,30){\makebox(0,0)[r]{$a_6$}}
\put(-5,20){\makebox(0,0)[r]{$a_7$}}
\put(-5,10){\makebox(0,0)[r]{$a_8$}}
\put(-5,0){\makebox(0,0)[r]{$a_9$}}
\put(10,-10){\makebox(0,0)[t]{$a$}}
\end{picture}
\hspace{10ex}
\begin{picture}(20,90)(0,-10)
\multiput(0,0)(0,10){9}{
\multiput(0,0)(10,0){3}{\circle*{1.5}}
}
\multiput(0,0)(0,10){6}{
\put(0,0){\line(1,0){20}}
}
\put(0,80){\line(1,0){20}}
\put(0,70){\line(1,0){10}}
\put(0,60){\line(1,0){10}}
\put(10,70){\line(1,-1){10}}
\put(10,60){\line(1,1){10}}

\put(-5,80){\makebox(0,0)[r]{$b_1$}}
\put(-5,70){\makebox(0,0)[r]{$b_2$}}
\put(-5,60){\makebox(0,0)[r]{$b_3$}}
\put(-5,50){\makebox(0,0)[r]{$b_4$}}
\put(-5,40){\makebox(0,0)[r]{$b_5$}}
\put(-5,30){\makebox(0,0)[r]{$b_6$}}
\put(-5,20){\makebox(0,0)[r]{$b_7$}}
\put(-5,10){\makebox(0,0)[r]{$b_8$}}
\put(-5,0){\makebox(0,0)[r]{$b_9$}}
\put(10,-10){\makebox(0,0)[t]{$b$}}
\end{picture}
\hspace{10ex}
\begin{picture}(20,90)(0,-10)
\multiput(0,0)(0,10){9}{
\multiput(0,0)(10,0){3}{\circle*{1.5}}
}
\multiput(0,0)(0,10){4}{
\put(0,0){\line(1,0){20}}
}
\multiput(0,60)(0,10){3}{\line(1,0){20}}
\put(0,50){\line(1,0){10}}
\put(0,40){\line(1,0){10}}
\put(10,50){\line(1,-1){10}}
\put(10,40){\line(1,1){10}}

\put(-5,80){\makebox(0,0)[r]{$c_1$}}
\put(-5,70){\makebox(0,0)[r]{$c_2$}}
\put(-5,60){\makebox(0,0)[r]{$c_3$}}
\put(-5,50){\makebox(0,0)[r]{$c_4$}}
\put(-5,40){\makebox(0,0)[r]{$c_5$}}
\put(-5,30){\makebox(0,0)[r]{$c_6$}}
\put(-5,20){\makebox(0,0)[r]{$c_7$}}
\put(-5,10){\makebox(0,0)[r]{$c_8$}}
\put(-5,0){\makebox(0,0)[r]{$c_9$}}
\put(10,-10){\makebox(0,0)[t]{$c$}}
\end{picture}
\hspace{10ex}
\begin{picture}(20,90)(0,-10)
\multiput(0,0)(0,10){9}{
\multiput(0,0)(10,0){3}{\circle*{1.5}}
}
\multiput(0,0)(0,10){4}{
\put(0,0){\line(1,0){20}}
}
\put(0,80){\line(1,0){20}}
\put(0,70){\line(1,0){10}}
\put(0,60){\line(1,0){10}}
\put(10,70){\line(1,-1){10}}
\put(10,60){\line(1,1){10}}
\put(0,50){\line(1,0){10}}
\put(0,40){\line(1,0){10}}
\put(10,50){\line(1,-1){10}}
\put(10,40){\line(1,1){10}}

\put(-5,80){\makebox(0,0)[r]{$d_1$}}
\put(-5,70){\makebox(0,0)[r]{$d_2$}}
\put(-5,60){\makebox(0,0)[r]{$d_3$}}
\put(-5,50){\makebox(0,0)[r]{$d_4$}}
\put(-5,40){\makebox(0,0)[r]{$d_5$}}
\put(-5,30){\makebox(0,0)[r]{$d_6$}}
\put(-5,20){\makebox(0,0)[r]{$d_7$}}
\put(-5,10){\makebox(0,0)[r]{$d_8$}}
\put(-5,0){\makebox(0,0)[r]{$d_9$}}
\put(10,-10){\makebox(0,0)[t]{$d$}}
\end{picture}
\end{center}
\end{figure}

\begin{claim}
If $a\in\mathfrak{X}$, then $\mathfrak{X}_{\alpha(a)}\subseteq \mathfrak{X}$.
\end{claim}

\begin{proof}[Proof of Claim:]
Choose $a\in\mathfrak{X}$. Note that if $b$ is obtained from $a$ by swapping two elements in blocks of $a$ other than $\alpha(a)$, then $b\in\mathfrak{X}$ and $a\sim b$, so by Claim~\ref{llk} $\alpha(a)=\alpha(b)$. So if $b^0,b^1\ldots,b^n$ is a sequence of relations with $b^0=a$ and each obtained from the previous by switching two elements not contained in the block $\alpha(a)$, then $\alpha(b^i) = \alpha(a)$ and $b^i\in\mathfrak{X}$ for each $i\leq n$. We claim that any $c$ having $\alpha(a)$ as a block can be obtained as $b^n$ for some such sequence $b^0,\ldots,b^n$ beginning with $b^0=a$. Suppose not. Then among all elements $b^n$ obtained by such a sequence, choose one, say $b$, that agrees with $c$ on as many blocks as possible. 
Suppose the blocks of $b$ are $b_1,\ldots,b_9$, the blocks of $c$ are $c_1,\ldots,c_9$, that $b,c$ agree on blocks $b_1,\ldots,b_k$ and that $b_1=c_1=\alpha(a)$. We may also assume that block $b_{k+1}$ is the block of $b$ having the most elements in common with $c_{k+1}$. 
Note $b_{k+1}$ and $c_{k+1}$ have either one or two elements in common, and the other elements in $c_{k+1}$ are in blocks $b_{k+1},\ldots,b_9$ since $b_i=c_i$ for $i\leq k$. 
If $b_{k+1}$ and $c_{k+1}$ agree on two elements, we may assume the third element of $c_{k+1}$ in the block $b_{k+2}$. Then by an appropriate swapping of elements in the blocks $b_{k+1},b_{k+2}$ of $b$, we can form another term $b^{n+1}$ to put at the end of our sequence that agrees with $c$ on $k+1$ blocks. If $b_{k+1}$ contains only one element of $c_{k+1}$, then we may assume the other two elements of $c_{k+1}$ are in $b_{k+2},b_{k+3}$. We can then extend our sequence by forming $b^{n+1},b^{n+2}$ first by swapping two elements in $b_{k+1},b_{k+2}$, then performing one more swap to put the element from $b_{k+3}$ into the block to form $c_{k+1}$. 
\end{proof}

As $\mathfrak{X}_{\alpha(a)}\subseteq \mathfrak{X}$ and both have $24!/(8!)(3!)^8$ elements, they are equal. 
\end{proof}

This result will allow us to transfer an automorphism of the poset of regular equivalence relations to a certain kind of permutation of the set of small blocks. The key point is that for an automorphism $\Lambda$ of $\ts{Req}~X$, Lemma~\ref{sue} shows that for a small block $\alpha$, the image $\Lambda[\mathfrak{X}_\alpha]$ is a set $\mathfrak{X}_\beta$ for some small block $\beta$. 

\begin{defn}
Let $\ts{Block}~X=\{\alpha:\alpha\mbox{ is a small block of }X\}$. We then define a permutation $\rho$ of this set to be special if for all small blocks $\alpha,\beta,\gamma$
\vspace{1ex}
\begin{enumerate}
\item $|\,\alpha\cap\beta\,| = |\,\rho(\alpha)\cap\rho(\beta)\,|$. 
\item $\gamma\subseteq\alpha\cup\beta$ iff $\rho(\gamma)\subseteq\rho(\alpha)\cup\rho(\beta)$.
\end{enumerate}
\label{ka}
\end{defn}

One easily sees that the special permutations of $\ts{Block}~\!X$ form a subgroup of its permutation group. 

\begin{prop}
There is an embedding $\ts{B}$ of the automorphism group of $\ts{Req}~X$ into the group of special permutations of $\ts{Block}~X$ taking $\Lambda$ to $\Lambda_{\ts{B}}$ where 
\vspace{1ex}
$$\Lambda_{\ts{B}}(\alpha)=\beta \quad\mbox{iff}\quad\Lambda[\mathfrak{X}_\alpha]=\mathfrak{X}_\beta.$$
\label{d}
\end{prop}

\begin{proof}
Suppose $\alpha$ is a small block and $\Lambda$ is an automorphism of $\ts{Req}~X$. The set $\mathfrak{X}_\alpha$ satisfies the four conditions in Lemma~\ref{sue}. These conditions involve the notion of a collapse, the numbers of collapses containing an element of $\mathfrak{X}_\alpha$ that are contained in $\mathfrak{X}_\alpha$ or nearly disjoint from this set, whether certain collapses share a block, and a statement involving numbers of upper bounds. Lemma~\ref{2collapse} shows the image of a collapse under an automorphsim is a collapse, and the definition of two collapses sharing a block is clearly preserved under automorphisms. Thus the image $\Lambda[\mathfrak{X}_\alpha]$ satisfies the conditions of Lemma~\ref{sue}, hence is $\mathfrak{X}_\beta$ for some small block $\beta$. The definition of $\mathfrak{X}_\beta$ shows the choice of $\beta$ is unique. 

These comments show that for each automorphism $\Lambda$ of $\ts{Req}~X$, there is a map $\Lambda_{\ts{B}}$ from $\ts{Block}~X$ to itself, temporarily called an endomorphism of this unstructured set, defined by $\Lambda_{\ts{B}}(\alpha)=\beta$ iff $\Lambda[\mathfrak{X}_\alpha]=\mathfrak{X}_\beta$. So there is a map $\ts{B}$ from the automorphism group of $\ts{Req}~X$ to the endomorphism monoid of the set of small blocks. 
Clearly $\ts{B}$ preserves composition and the identity map, so $\ts{B}$ is a homomorphism from the automorphism group of $\ts{Req}~X$ to the permutation group of the small blocks. To see $\ts{B}$ is one-one, suppose $\Lambda\neq id$. Then $\Lambda(a)\neq a$ for some small relation $a$, so there is a block $\alpha$ of $a$ that is not a block of $\Lambda(a)$. It follows that $\Lambda[\mathfrak{X}_\alpha]\neq\mathfrak{X}_\alpha$. So $\ts{B}(\Lambda)\neq id$. Thus $\ts{B}$ is an embedding. It remains to show each $\Lambda_{\ts{B}}$ is a special permutation. 

\begin{claim}
$|\alpha\cap\beta| = |\Lambda_{\ts{B}}(\alpha)\cap\Lambda_{\ts{B}}(\beta)|$.
\end{claim}

\begin{proof}[Proof of Claim:]
Note $\alpha$ and $\beta$ are disjoint iff $\mathfrak{X}_\alpha$ and $\mathfrak{X}_\beta$ are not disjoint. It follows that $|\,\alpha\cap\beta\,|=0$ iff $|\,\Lambda_{\ts{B}}(\alpha)\cap\Lambda_{\ts{B}}(\beta)\,|=0$. As $\Lambda_{\ts{B}}$ is a permutation, we have $\alpha=\beta$ iff $\Lambda_{\ts{B}}(\alpha)=\Lambda_{\ts{B}}(\beta)$ and the case of an intersection in 3 elements follows. Once we know that $\alpha$ and $\beta$ intersect in either 1 or 2 elements, we can distinguish the cases as follows: $\alpha$ and $\beta$ intersect in 2 elements iff for any small $a$ having $\alpha$ as a block, there is a small $b$ having $\beta$ as a block, so that $\ts{Trcl}(a\cup b)$ has 1 block of 6 and 7 blocks of 3. By Lemma~\ref{countubs} this condition on $\ts{Trcl}(a\cup b)$ occurs iff $a,b$ have 70 large upper bounds. So it follows that $|\alpha\cap\beta| = 2$ iff $|\Lambda_{\ts{B}}(\alpha)\cap\Lambda_{\ts{B}}(\beta)| = 2$. The statement for an intersection in one element follows by elimination. 
\end{proof}

\begin{claim}
$\gamma\subseteq\alpha\cup\beta$ iff $\Lambda_{\ts{B}}(\gamma)\subseteq\Lambda_{\ts{B}}(\alpha)\cup\Lambda_{\ts{B}}(\beta)$.
\end{claim}

\begin{proof}[Proof of Claim:]
We show $\gamma\subseteq\alpha\cup\beta$ implies $\Lambda_{\ts{B}}(\gamma)\subseteq\Lambda_{\ts{B}}(\alpha)\cup\Lambda_{\ts{B}}(\beta)$. The other direction follows by the same result for the inverse $\Lambda^{-1}$. We aregue the contrapositive. Suppose $x\in\Lambda_{\ts{B}}(\gamma)$ and $x$ is not in $\Lambda_{\ts{B}}(\alpha)\cup\Lambda_{\ts{B}}(\beta)$. Choose two other elements $y,z$ not belonging to $\Lambda_{\ts{B}}(\alpha)\cup\Lambda_{\ts{B}}(\beta)$, and let $\delta$ be the small block with $\Lambda_{\ts{B}}(\delta) = \{x,y,z\}$. Then $\Lambda_{\ts{B}}(\delta)$ is disjoint from both $\Lambda_{\ts{B}}(\alpha)$ and $\Lambda_{\ts{B}}(\beta)$, so $\delta$ is disjoint from $\alpha$ and $\beta$. But $\Lambda_{\ts{B}}(\delta)$ intersects $\Lambda_{\ts{B}}(\gamma)$ non-trivially, so $\gamma$ contains an element of $\delta$, hence $\gamma$ is not contained in $\alpha\cup\beta$. 
\end{proof}

This concludes the proof of the proposition. 
\end{proof}

We turn to our final step, associating to a special permutation $\rho$ of $\ts{Block}~X$ a permutation of $X$. The key is the following lemma. 

\begin{lemma}
Let $\rho$ be a special permutation of $\ts{Block}~X$. If $\alpha,\beta,\gamma$ are small blocks with $\alpha\cap\beta=\{p\}$ and $\rho(\alpha)\cap\rho(\beta) = \{x\}$, then $p\in\gamma\Rightarrow x\in\rho(\gamma)$.
\label{ln}
\end{lemma}

\begin{proof}
Using symmetry, there are four cases. (1) $\gamma$ intersects both $\alpha$ and $\beta$ in two elements. (2) $\gamma$ intersects one of $\alpha,\beta$ in two elements, and the other in one element. (3) $\gamma$ intersects both $\alpha,\beta$ in one element. (4) $\gamma$ is equal to one of $\alpha,\beta$. We begin with a claim that establishes case (1) and that will be used repeatedly to establish the other cases. 

\begin{claim}
For any small blocks $\alpha,\beta,\gamma$, if $\gamma$ intersects both $\alpha,\beta$ in two elements and $\rho(\alpha)=\{x,y_1,y_2\}$, $\rho(\beta)=\{x,z_1,z_2\}$, then $\rho(\gamma)$ is $\{x,y_i,z_j\}$ for some $i,j\in\{1,2\}$. 
\label{gn}
\end{claim}

\begin{proof}[Proof of Claim:]
This is obvious from the fact that special permutations preserve the cardinality of intersections and the pigeonhole principle.
\end{proof}

\pagebreak[3]

\begin{claim}
Suppose $p,q_1,q_2,r_1,r_2$ are distinct elements of $X$. There there are unique elements $x,y_1,y_2,z_1,z_2$ with 
\vspace{1ex}

{\em 
\begin{tabular}{lllll}
(1)\,\,\, $\rho\{p,q_1,q_2\} \,\,\,= \,\,\{x,y_1,y_2\}$ &&&& (2)\,\, \, $\rho\{p,r_1,r_2\} \,\,\,=\,\, \{x,z_1,z_2\}$\\
(3)\,\,\, $\rho\{p,q_1,r_1\} \,\,\, = \,\,\{x,y_1,z_1\}$ &&&& (4) \,\,\, $\rho\{p,q_2,r_2\} \,\,\,=\,\, \{x,y_2,z_2\}$\\
(5)\,\,\, $\rho\{p,q_1,r_2\} \,\,\, =\,\, \{x,y_1,z_2\}$ &&&& (6) \,\,\, $\rho\{p,q_2,r_1\} \,\,\,=\,\, \{x,y_2,z_1\}$\\
(7)\,\,\, $\rho\{q_1,q_2,r_1\} \,=\,\, \{y_1,y_2,z_1\}$ &&&& (8) \,\,\, $\rho \{q_1,q_2,r_2\} \,=\,\, \{y_1,y_2,z_2\}$\\
(9)\,\,\, $\rho \{q_1,r_1,r_2\} \,=\,\, \{y_1,z_1,z_2\}$ &&&& (10) \hspace{.1ex} $\rho\{q_2,r_1,r_2\} \,=\,\, \{y_2,z_1,z_2\}$
\end{tabular}
}
\label{fb}
\end{claim}
\vspace{1ex}

\begin{proof}[Proof of Claim:] 
We have $\rho\{p,q_1,q_2\}$ and $\rho\{p,r_1,r_2\}$ intersect in an element. So items (1) and (2) are clear. For item (3) we have $\{p,q_1,r_1\}$ intersects both of $\{p,q_1,q_2\}$ and $\{p,r_1,r_2\}$ in two elements, so by Claim~\ref{gn} we have $\rho\{p,q_1,r_1\}$ equals $\{x,y_i,z_j\}$ for some choice of $i,j$. At this point we are free to choose the numbering of $y_1,y_2$ and $z_1,z_2$ as we please, and we number them to make (3) hold. We again use Claim~\ref{gn} to obtain $\rho\{p,q_2,r_2\}$ is of the form $\{x,y_i,z_j\}$ for some $i,j$, and as $\rho\{p,q_1,r_1\}$ and $\rho\{p,q_2,r_2\}$ must intersect only in one element, it must be as indicated in (4). 
For (5), the argument for (3) and (4) shows $\rho\{p,q_1,r_2\} = \{x,y_i,z_j\}$ for some choice of $i,j$, and it cannot be either of the results in (3) or (4).  As $\{p,q_1,r_2\}$ is contained in $\{p,r_1,r_2\}\cup\{p,q_1,r_1\}$, by the definition of special in Definition~\ref{ka} we have that $\rho\{p,q_1,r_2\}\subseteq \{x,y_1,z_1,z_2\}$. So $\rho\{p,q_1,r_2\}=\{x,y_1,z_2\}$. For (6) $\rho\{p,q_2,r_1\}$ is again of the form $\{x,y_i,z_j\}$ and this is the only choice remaining. 
For (7), the definition of special shows $\rho\{q_1,q_2,r_1\}\subseteq\rho\{p,q_1,q_2\}\cup\rho\{p,q_1,r_1\}$, therefore is a 3-element subset of $\{x,y_1,y_2,z_1\}$. The only 3-element subset of this set not already in our list is $\{y_1,y_2,z_1\}$. The arguments for (8) through (10) are similar. 
\end{proof}

We now return to the proof of Lemma~\ref{ln}. Assume that $\alpha\cap\beta=\{p\}$, and $\rho(\alpha)\cap\rho(\beta)=\{x\}$, and that $\gamma$ is a small block with $p\in\gamma$. We must show $x\in\rho(\gamma)$. Suppose $\alpha=\{p,q_1,q_2\}$ and $\beta=\{p,r_1,r_2\}$ and $x,y_1,y_2,z_1,z_2$ are as in Claim~\ref{fb}. 
We remarked that up to symmetry there were four cases, and proved in Claim~\ref{gn} case (1) where $\gamma$ intersected both $\alpha,\beta$ in two elements. Consider the other cases. 

Case (2) has $\gamma$ intersect $\alpha$ in two elements, and $\gamma$ intersect $\beta$ in one element. We may assume $\gamma=\{p,q_1,s\}$ for some $s$ not among $p,q_1,q_2,r_1,r_2$. Then $\rho(\gamma)$ intersects $\rho\{p,q_1,q_2\} = \{x,y_1,y_2\}$ in 2 elements, and is disjoint from $\rho\{q_2,r_1,r_2\} = \{y_2,z_1,z_2\}$. So $x$ belongs to $\rho(\gamma)$. Case (3) has $\gamma$ intersect each of $\alpha,\beta$ in one element. Assume $\gamma=\{p,s,t\}$ for some elements $s,t$ not among $p,q_1,q_2,r_1,r_2$. Then $\rho(\gamma)$ intersects $\rho\{p,q_1,q_2\} = \{x,y_1,y_2\}$ in 1 element, and is disjoint from $\rho\{q_1,q_2,r_1\} = \{y_1,y_2,z_1\}$. So again, $x$ belongs to $\rho(\gamma)$. The final case (4), where $\gamma$ equals one of $\alpha,\beta$ is trivial. 
\end{proof}

\begin{prop}
There is a one-one group homomorphism $\ts{P}$ from the group of special permutations of $\ts{Block}~X$ to the permutation group of $X$ such that

\[\mbox{$\alpha\cap\beta=\{p\}$ and $\rho(\alpha)\cap\rho(\beta)=\{x\}\,\, \Rightarrow \,\,(\ts{P}\rho)(p) = x$.}\]
\label{P}
\label{e}
\end{prop}
\vspace{-3ex}

\begin{proof}
Suppose $\rho$ is a special permutation. For any $p\in X$ we can find small blocks $\alpha,\beta$ with $\alpha\cap\beta=\{p\}$, and as $\rho$ is special, $\rho(\alpha)\cap\rho(\beta)$ is a singleton $\{x\}$. Suppose $\alpha'$ and $\beta'$ are other small blocks with $\alpha'\cap\beta'=\{p\}$. Then as $\rho$ is special, $\rho(\alpha')\cap\rho(\beta')$ is a singleton. But $p\in\alpha'$ and $p\in\beta'$, and Lemma~\ref{ln} shows $x\in\rho(\alpha')$ and $x\in\rho(\beta')$. So $\rho(\alpha')\cap\rho(\beta')=\{x\}$. Thus, we can define a map $\ts{P}\rho$ by setting $(\ts{P}\rho)(p)=x$ if there are small $\alpha$, $\beta$ with $\alpha\cap\beta=\{p\}$ and $\rho(\alpha)\cap\rho(\beta)=\{x\}$. 

One sees that $\ts{P}$ preserves composition and the identity map, so is a group homomorphism from the special permutations of $\ts{Block}~X$ to the permutation group of $X$. To see it is one-one, suppose $\rho$ is a special permutation that is not the identity. There there is a small block $\alpha$ with $\rho(\alpha)\neq\alpha$, hence some $p\in\alpha$ with $p\not\in\rho(\alpha)$. Find $\beta$ with $\alpha\cap\beta=\{p\}$. Then $\rho(\alpha)\cap\rho(\beta)\neq\{p\}$, so $(\ts{P}\rho)(p)\neq p$. Thus $\ts{P}\rho\neq id$. 
\end{proof}

We recall our many steps. Theorem~\ref{a} states that $\Gamma$ is a homomorphism from the permutation group of $X$ to the automorphism group of the $\ts{omp}$ $\ts{Fact}~X$, and Proposition~\ref{b} shows that $\Gamma$ is an embedding. Theorem~\ref{c} shows that $\ts{R}$ is a group embedding of the automorphism group of $\ts{Fact}~X$ into the automorphism group of the poset $\ts{Req}~X$ of regular equivalence relations on $X$. Proposition~\ref{d} shows there is an embedding $\ts{B}$ of the automorphism group of $\ts{Req}~X$ into the group of special permutations of the set $\ts{Block}~X$ of all small blocks of $X$. Finally, Proposition~\ref{e} shows there is a group embedding $\ts{P}$ of the group of special permutations of $\ts{Block}~X$ into the permutation group of $X$. 

\begin{thm}
Each of the following groups are isomorphic:
\vspace{1ex}
\begin{enumerate}
\item The automorphism group of $\ts{Fact}~X$.
\item The automorphism group of the poset $\ts{Req}~X$.
\item The group of special permutations of $\ts{Block}~X$.
\item The permuation group of $X$. 
\end{enumerate}
\vspace{1ex}

\noindent Further, the maps $\Gamma$, $\ts{R}$, $\ts{B}$, and $\ts{P}$ are isomorphisms. 
\end{thm}

\begin{proof}
We have shown each of these groups are isomorphic to subgroups of another via the indicated maps, and the groups involved are finite. 
\end{proof}

\section{Further remarks}

In this section, we discuss additional results, and a number of open problems. We begin with the following result of Chevalier \cite{Chevalier3}.

\begin{prop}
For a vector space $V$, each order-automorphism of $\ts{Fact}~V$ is an \ts{omp} automorphism. 
\end{prop}

This result came from Chevalier's proof characterizing automorphisms of $\ts{Fact}~V$. The key step in showing that automorphisms essentially work componentwise, is that two atoms of $\ts{Fact}~V$ have distinct coatom upper bounds iff they have the same first or second spots. This property is clearly preserved by any order-automorphism as it does not involve the orthocomplementation. For finite vector spaces, one can show somewhat more. 

\begin{prop}
For a finite vector space $V$, there is a unique orthocomplementation on $\ts{Fact}~V$ compatible with its order structure. 
\end{prop}

We will not give the proof, but remark that the key ingredient is Baer's \cite{Baerpolarity} result that any polarity on a finite projective plane has an absolute point. This finiteness condition is needed as it is easily seen that any orthocomplementation on the subspace lattice $\ts{Sub}~V$ gives an orthocomplementation on $\ts{Fact}~V$. So $\ts{Fact}~\Bbb{R}^3$ will admit many orthocomplementations \cite{Birkhoff,ChevalierOC}. We do not know if these auxiliary orthocomplementations are orthomodular. 

\begin{problem}
If $X$ is a set that is not small, are all order-automosphisms of $\ts{Fact}~X$ \ts{omp} automorphisms? Is $\ts{Fact}~X$ uniquely orthocomplemented? 
\end{problem}

Recall that a state \cite{Kalmbach,Ptak} on an \ts{omp} $P$ is a map $\sigma$ from $P$ to the real unit interval that maps $0$ to $0$, $1$ to $1$, and satisfies $x\perp y\Rightarrow \sigma(x\oplus y) = \sigma(x) + \sigma(y)$. For any \ts{omp} whose blocks are all finite, states correspond to maps from the atoms to the real unit interval such that the values on the atoms of any block sum to 1. Clearly, on any \ts{omp} whose blocks all have the same finite number $n$ of atoms, such as $\ts{Fact}~V$ for a finite-dimensional vector space $V$ or $\ts{Fact}~X$ for a finite set $X$, there is a state taking value $1/n$ on each atom. 

\begin{prop}
$\ts{Fact}~\Bbb{Z}_2^3$ has only the state taking value $1/3$ on the atoms.
\label{vbv}
\end{prop}

\begin{proof}
Recall that this \ts{omp} has 28 atoms and 28 blocks. By hand, one can enumerate these atoms as $x_1,\ldots,x_{28}$ and these blocks $B_1,\ldots,B_{28}$, and set up an incidence matrix $A$ with the columns representing atoms, the rows the blocks, and spot $A_{ij}$ being either $0$ or $1$ depending on whether atom $x_j$ lies in block $B_i$. This matrix $A$ will have each row consisting of three 1's with the other entries $0$'s. Then states correspond to solutions of the system of equations $A\vec{x}=1$ where $1$ is the column matrix of all $1's$. Using a computer algebra system, the determinant of this matrix is non-zero. So the system has a unique solution, and this is the state taking constant value $1/3$.
\end{proof}

One can easily produce, or find on the web \cite{Moorehouse}, point-line incidence matrices for other small projective planes over finite fields. Automating the above construction of the atom-block incidence matrix as in the proof of Proposition~\ref{vbv} and using a software package to solve the system of equations, one obtains the following. 

\begin{prop}
For $V$ a 3-dimensional vector space over a field with $2,3,4,5,7$ elements, then $\ts{Fact}~V$ has only the constant state taking values 1/3 on the atoms. 
\end{prop}

If $X$ is a set with $p^3$ elements, for $p$ prime, then each block of $\ts{Fact}~X$ is a block of some subalgebra isomorphic to $\ts{Fact}~\Bbb{Z}_p^3$. The following is then immediate. 

\begin{prop}
For $X$ a set with $2^3,3^3,5^3$ or $7^3$ elements, $\ts{Fact}~X$ has only the state taking constant value 1/3 on the atoms. 
\end{prop}

For a vector space $V$ of dimension $2$, or a finite set $X$ where $|X|$ has only 2 prime factors, the \ts{omp}s $\ts{Fact}~V$ and $\ts{Fact}~X$ are of the form $\ts{Mo}_n$ for some $n$. These are known to have infinitely many states. 

\begin{problem}
If $V$ is a finite-dimensional vector space of dimension at least 3, and if $X$ is a finite set whose cardinality has at least 3 prime factors, is the state taking constant value on the atoms the only state on $\ts{Fact}~V$ or $\ts{Fact}~X$?
\end{problem}

We next turn our attention to group-valued measures that play an important role in the theory of unigroups \cite{Foulis} of \ts{omp}s. 

\begin{defn}
A group-valued measure on an \ts{omp} $P$ is a map $\sigma$ from $P$ into an abelian group $G$ that satisfies $\sigma(0)=0$ and $x\perp y\Rightarrow \sigma(x\oplus y)=\sigma(x)+\sigma(y)$. 
\end{defn}

If each block of an \ts{omp} $P$ has the same finite number of atoms, then for any abelian group $G$ and element $a\in G$ there is group-valued measure on $P$ taking value $a$ on each atom of $P$. We call these constant measures. Using techniques from Proposition~\ref{vbv} and a linear algebra package to solve equations over finite fields, we obtain the following. 

\begin{prop}
Suppose $k\in\{2,3,4,5,7\}$,  $p\in\{2,3,5,7\}$, and consider $\Bbb{Z}_p$-valued measures on $\ts{Fact}~V$ where $V$ is a 3-dimensional vector space over the $k$-element field. If $p$ does not divide $k$, then all such measures are constant, and if $p$ divides $k$, then there are $pk^8$ such measures.
\end{prop}

When considering group valued measures on $\ts{Fact}~X$ for a finite set $X$, we know the situation for an 8-element set as $\ts{Fact}~X$ is a horizontal sum of copies of $\ts{Fact}~\Bbb{Z}_2^3$, so has a huge number of $\Bbb{Z}_2$-valued states. For other cases of finite sets of prime power cardinality, we know no more than the obvious fact that $\Bbb{Z}_p$-valued states on $\ts{Fact}~X$ will be constant when $p$ does not divide $|X|$. The appearance of a power of 8 in the above result is perhaps related to the following question. 

\begin{problem}
Determine the unigroup of $\ts{Fact}~V$ for a finite-dimensional vector space $V$, and the unigroup of $\ts{Fact}~X$ for a finite set $X$. 
\end{problem}

The main open problem of this paper lies in extending results for the 27-element set to the general case. We know there is a trouble with the case of a set with $2^3$ elements, and have no intuition about the behavior for other {\em small} cases such as a set with $2^2\times 3$, or $2\times 3^2$ or $2^n$ elements. But we suspect that if $X$ has at least three prime factors bigger than 2, then the situation is well-behaved. We record this below. 

\begin{problem}
If $X$ is a set of sufficiently large size, is the embedding $\Gamma$ of the permutation group of $X$ into the automorphism group of $\ts{Fact}~X$ an isomorphism?
\end{problem}

\end{document}